\documentclass[11pt, reqno]{amsart}
\usepackage{amssymb, amsmath, amsthm,amscd,mathabx,amsxtra}
\usepackage[colorlinks,plainpages,backref=page,urlcolor=blue]{hyperref}
\hypersetup{colorlinks=true, citecolor=magenta}
\usepackage[alphabetic,lite]{amsrefs}
\usepackage{color}
\usepackage{mathtools}
\usepackage{verbatim}
\usepackage{tikz}
\usetikzlibrary{cd,calc,arrows}
\usepackage{mathrsfs}
\usepackage{upgreek}
\usepackage{cases}
\usepackage[shortlabels]{enumitem}
\usepackage{colonequals}
\usepackage{eqparbox}
\usepackage{array}
\usepackage{setspace}
\usepackage{breakurl}
\usepackage{xurl}
\hypersetup{breaklinks=true,  colorlinks=true,  pdfusetitle=true}
\makeatletter
 \def\l@subsection{\@tocline{2}{0pt}{4pc}{6pc}{}}
\def\l@subsubsection{\@tocline{3}{0pt}{8pc}{8pc}{}}
 \makeatother

\newcommand{\defi}[1]{{\emph{#1}}}

\renewcommand{\a}{\alpha}

\renewcommand{\b}{\beta}

\newcommand{\ds}{\displaystyle}

\newcommand{\bw}{\bigwedge}

\newcommand{\Gr}{{\bf G}}

\newcommand{\onto}{\twoheadrightarrow}
\newcommand{\oo}{\otimes}
\newcommand{\pd}{\partial}
\renewcommand{\P}{{\bf P}}

\newcommand{\Rproj}{{\bf R}}

\newcommand{\W}{{\mc W}}

\newcommand{\Z}{\mathbb{Z}}
\newcommand{\Q}{\mathbb{Q}}

\newcommand{\C}{\mathbb{C}}
\renewcommand{\k}{\Bbbk}

\newcommand{\RR}{\mathcal{R}}

\newcommand{\PP}{\mathcal{P}}
\newcommand{\EE}{\mathcal{E}}
\newcommand{\FF}{\mathcal{F}}

\newcommand{\SU}{\mathcal{SU}}

\newcommand{\Ann}{\operatorname{Ann}}

\newcommand{\Ext}{\operatorname{Ext}}

\newcommand{\Hom}{\operatorname{Hom}}

\newcommand{\rk}{\operatorname{rank}}

\newcommand{\Sp}{\operatorname{Sp}}
\newcommand{\Aut}{\operatorname{Aut}}
\newcommand{\Spec}{\operatorname{Spec}}
\newcommand{\Sym}{\operatorname{Sym}}

\newcommand{\Hilb}{\operatorname{Hilb}}

\newcommand{\Fitt}{\operatorname{Fitt}}
\newcommand{\Grass}{\operatorname{Gr}}

\newcommand{\End}{\operatorname{End}}
\newcommand{\Pic}{\operatorname{Pic}}

\newcommand{\coker}{\operatorname{coker}}
\newcommand{\Spn}{\operatorname{Span}}
\newcommand{\TC}{\operatorname{TC}}

\newcommand{\gr}{\operatorname{gr}}

\newcommand{\pr}{\operatorname{pr}}
\newcommand{\rank}{\operatorname{rank}}
\newcommand{\im}{\operatorname{im}}
\newcommand{\id}{\operatorname{id}}

\newcommand{\spn}{\operatorname{Span}}

\newcommand{\bb}[1]{\mathbb{#1}}

\renewcommand{\rm}[1]{\textrm{#1}}
\newcommand{\mc}[1]{\mathcal{#1}}
\newcommand{\mf}[1]{\mathfrak{#1}}
\newcommand{\ol}[1]{\overline{#1}}
\newcommand{\op}[1]{\operatorname{#1}}

\def\lra{\longrightarrow}

\newcommand{\surj}{\twoheadrightarrow}
\newcommand{\inj}{\hookrightarrow}
\newcommand{\isom}{\xrightarrow{
   \,\smash{\raisebox{-0.3ex}{\ensuremath{\scriptstyle\simeq}}}\,}}
\newcommand{\bwedge}{\mbox{\normalsize $\bigwedge$}}

\def\dot{\mathchar"013A}
\newcommand{\hdot}{{\raise1pt\hbox to0.35em{\!\Huge $\dot$}}}
\definecolor{dkgreen}{RGB}{0,100,0}
\definecolor{dkbrown}{RGB}{139,69,19}

\newtheorem{theorem}{Theorem}[section]
\newtheorem*{theorem*}{Theorem}
\newtheorem*{problem*}{Problem}
\newtheorem{lemma}[theorem]{Lemma}
\newtheorem{conjecture}[theorem]{Conjecture}
\newtheorem{proposition}[theorem]{Proposition}
\newtheorem{corollary}[theorem]{Corollary}
\newtheorem*{corollary*}{Corollary}
\newtheorem*{claim*}{Claim}

\theoremstyle{definition}
\newtheorem{definition}[theorem]{Definition}
\newtheorem*{definition*}{Definition}
\newtheorem{example}[theorem]{Example}
\newtheorem{question}[theorem]{Question}
\newtheorem{remark}[theorem]{Remark}
\newtheorem*{remark*}{Remark}
\newtheorem*{notation*}{Notation}
\newtheorem*{ack}{Acknowledgments}

\numberwithin{equation}{section}

\newcommand{\abs}[1]{\lvert#1\rvert}

\begin{document}

\title{Reduced resonance schemes and Chen ranks}

\author[M. Aprodu]{Marian Aprodu}
\address{Marian Aprodu: Simion Stoilow Institute of Mathematics
\hfill \newline\texttt{}
 \indent P.O. Box 1-764,
RO-014700 Bucharest, Romania, and \hfill  \newline
\indent  Faculty of Mathematics and Computer Science, University of Bucharest,  Romania}
\email{{\tt marian.aprodu@imar.ro}}

\author[G. Farkas]{Gavril Farkas}
\address{Gavril Farkas: Institut f\"ur Mathematik,   Humboldt-Universit\"at zu Berlin \hfill \newline
\indent Unter den Linden 6,
10099 Berlin, Germany}
\email{{\tt farkas@math.hu-berlin.de}}

\author[C. Raicu]{Claudiu Raicu}
\address{Claudiu Raicu: Department of Mathematics,
University of Notre Dame \hfill \newline
\indent 255 Hurley Notre Dame, IN 46556, USA, and \hfill\newline
\indent Simion Stoilow Institute of Mathematics, \hfill\newline
\indent  P.O. Box 1-764, RO-014700 Bucharest, Romania}
\email{{\tt craicu@nd.edu}}

\author[A. Suciu]{Alexander I. Suciu}
\address{Alexander I. Suciu: Department of Mathematics,
Northeastern University \hfill \newline
\indent Boston, MA,
02115, USA}
\email{{\tt  a.suciu@northeastern.edu}}



\begin{abstract}
The resonance varieties are cohomological invariants that are studied in
a variety of topological, combinatorial, and geometric contexts. We discuss
their scheme structure in a general algebraic setting and introduce various
properties that ensure the reducedness of the associated projective resonance
scheme. We prove an asymptotic formula for the Hilbert series of the associated
Koszul module, then discuss applications to vector bundles on algebraic curves
and to Chen ranks formulas for finitely generated groups, with special emphasis
on K\"ahler and right-angled Artin groups.
\end{abstract}

\maketitle

\begin{spacing}{0.0}
\setcounter{tocdepth}{1}
\tableofcontents
\end{spacing}

\section{Introduction}
\label{sect:main}

The concept of resonance initially appeared in topology. For a reasonably nice topological
space $X$, its resonance variety $\RR(X)$ is defined as the jump locus
\begin{equation}\label{eq:res_top}
\RR(X)\coloneqq \Bigl\{a\in H^1(X,\C): H^1\bigl(H^*(X,\C), \delta_a\bigr)\neq 0\Bigr\},
\end{equation}
where $\delta_a\colon H^i(X, \C)\rightarrow H^{i+1}(X, \C)$ is the
differential obtained by left-multiplication by $a$. Assuming $X$ is $1$-formal in
the sense of Sullivan, $\RR(X)$ is intimately related to the \emph{characteristic variety}
$\mathcal{V}(X)$ parametrizing rank $1$ local systems on $X$ with non-vanishing homology.
More precisely (see \cite{DPS-duke}), the resonance can be described as the tangent cone to the
characteristic variety, that is, $\RR(X)=\TC_1\bigl(\mathcal{V}(X)\bigr)$, which links the
resonance to the much studied theory of characteristic varieties. A set-theoretic description
of $\RR(X)$ for compact K\"ahler manifolds in terms of irrational pencils has been given
by Dimca--Papadima--Suciu \cite{DPS-duke}. For complements of hyperplane arrangements
a combinatorial (matroidal) description of the resonance has been found by Falk--Yuzvinsky \cite{FY},
building on previous work of Libgober--Yuzvinsky \cite{LY00} and Falk \cite{Fa97}. A purely algebraic
definition of the resonance variety has been put forward by Papadima--Suciu \cite{PS-crelle} and
linked to the the theory of Koszul modules. Important applications of this link, including a proof
of the Generic Green's Conjecture for syzygies of canonical curves in sufficiently high characteristic
have been obtained in \cites{AFPRW, AFPRW2}.

By its very definition as a jump locus, the resonance $\RR(X)$ carries a natural scheme structure
that has been however largely ignored for  some time. However, even in the much studied case
of hyperplane arrangements, this scheme structure comes to the fore in the formulation of
Suciu's Conjecture \cite{Su-conm} concerning the Chen ranks of the fundamental group of
a hyperplane arrangement, see also \cites{CSc-adv, SS-tams}. We aim to study systematically
the scheme-theoretic properties of resonance varieties in algebraic context, introduce several
natural scheme-theoretic properties that resonance varieties often satisfy  and explain how they
are naturally linked to other interesting concepts in the theory of K\"ahler groups, vector bundles
on algebraic varieties, or geometric group theory.

We begin by recalling the connection between resonance varieties and Koszul
modules following the set-up of \cites{AFPRW2, AFPRW}, or \cite{PS-crelle}.
Let $V$ be a finite-dimensional vector space over an algebraically closed
field $\k$ of characteristic $0$, and let $V^{\vee}$ be its $\k$-dual.
We let $S\coloneqq \Sym(V)$ denote the symmetric algebra on $V$.
For a linear subspace $K\subseteq \bigwedge^2 V$, we define
the \defi{Koszul module}\/ $W(V,K)$ as the middle homology of
the chain complex of graded $S$-modules
\begin{equation*}
\label{eq:intro-WVK}
\begin{tikzcd}[column sep=22pt]
K \oo S \ar[rr, "\left. \delta_2\right|_{K \oo S}"] && V\oo S(1) \ar[r, "\delta_1"] & S(2),
\end{tikzcd}
\end{equation*}
where $\delta_2\colon \bigwedge^2 V\otimes S\rightarrow V\otimes S(1)$ is the
Koszul differential, given by
$\delta_2(v_1\wedge v_2\otimes f)=v_2\otimes v_1 f-v_1\otimes v_2 f$,
whereas $\delta_1(v\otimes f)=v f$ is the multiplication map.
Note that $W(V,K)$ is a graded $S$-module. Its degree $q$ piece $W_q(V,K)$
can be canonically identified with the homology of the following complex
of finite dimensional $\k$-vector spaces,
\[
\begin{tikzcd}[column sep=22pt]
K \oo \Sym^q V \ar[rr, "\left. \delta_2\right|_{K\oo \Sym^q V}"]
&& V\oo \Sym^{q+1} V \ar[r, "\delta_1"] & \Sym^{q+2} V .
\end{tikzcd}
\]
The annihilator $\Ann W(V,K)$ of the $S$-module $W(V,K)$ defines a
{\em resonance scheme}\/ $\RR(V,K)$, whose set-theoretic support, also
denoted by $\RR(V, K)$, has been described in \cite{PS-crelle}  as being
\begin{equation}
\label{eq:def-resonance}
\RR(V,K)\coloneqq\Bigl\{a\in V^\vee  : \text{ there exists $b\in
V^\vee$ such that $a\wedge b\in K^\perp\setminus \{0\}$} \Bigr\}\cup \{0\},
\end{equation}
where $K^{\perp}$ is the kernel of the projection $\bwedge^2 V^{\vee} \surj K^{\vee}$.
We consider here the \emph{projectivized resonance}\/ scheme,
\[
\Rproj(V,K)\coloneqq \mbox{Proj}\bigl(S/\Ann \ W(V,K)\bigr).
\]

For a topological space $X$, we take $V\coloneqq H_1(X, \C)$
and we let $K\subseteq \bigwedge^2 V$ be the image of the map
$\partial_X\colon H_2(X, \C)\to \bwedge^2 H_1(X, \C)$ defined as
the dual of the usual cup product map on $H^1(X, \C)$. With this notation,
we recover the topological resonance $\mathcal{R}(X)=\mathcal{R}(V, K)$
as defined in \eqref{eq:res_top} and studied in \cites{AFPRW, DPS-duke, PS-crelle}.

We investigate in what follows  the geometry of the resonance schemes and introduce
several concepts inspired by the study of particular cases of resonance varieties in
topology and Hodge theory, which turn out to be often satisfied and which ensure
the reducedness of the projectivized resonance $\Rproj(V,K)$. More precisely, in
Section \ref{subsec:bundles} we discuss the resonance varieties associated to
vector bundles on curves, in Section \ref{subsect:Kahler} we treat the case of
K\"{a}hler groups, whereas Section \ref{sect:right-angled} is devoted to
right-angled Artin groups.

\subsection{Separable and isotropic resonance}
\label{intro:sep-isotropic}
We denote by $E\coloneqq \bigwedge V^{\vee}$ the exterior algebra of
$V^{\vee}$ and fix a linear subspace $\ol{V}^{\vee} \subseteq V^{\vee}$.
We say that  $\ol{V}^{\vee}$ is {\em isotropic}\/
(with respect to $K\subseteq \bwedge^2 V$) if $\bwedge^2 \ol{V}^{\vee}\subseteq K^{\perp}$.
The subspace $\ol{V}^{\vee}$ is said to be {\em separable} (with respect to
$K$) if
\begin{equation}
\label{eq:sep}
K^{\perp} \cap \langle \ol{V}^{\vee}\rangle_E \subseteq
\bwedge^2 \ol{V}^{\vee},
\end{equation}
where $\langle \ol{V}^{\vee}\rangle_E$ denotes the ideal of  the exterior algebra
$E$ which is generated by $\ol{V}^{\vee}$.
Finally, $\ol{V}^{\vee}$ is {\em strongly isotropic} if it is both separable and isotropic,
that is, when the following equality holds
\[
K^{\perp} \cap \langle \ol{V}^{\vee}\rangle_E =
\bwedge^2 \ol{V}^{\vee}.
\]
If all the irreducible components of the resonance variety $\RR(V,K)$
are linear subspaces of $V^{\vee}$, we say that $\RR(V,K)$
is {\em separable} (respectively, {\em isotropic}, or {\em strongly isotropic})
if each of those components of $\RR(V,K)$ are
separable (respectively, isotropic, or strongly isotropic).
We use the same terminology for the Koszul module $W(V,K)$,
respectively for the resonance scheme $\Rproj(V,K)$. A characterization
of strongly isotropic resonance $\RR(V,K)$ reminiscent of Petri type
theorems in algebraic geometry is provided in Lemma \ref{lem:separable-multiplication}.

These definitions, while new in this general algebraic context, are inspired by the
study of the topological resonance. For instance, if $X$ is a complex smooth
quasi-projective variety, then $\mathcal{R}(X)$ is linear, but not necessarily
isotropic. If the mixed Hodge structure on $H^1(X,\C)$ is pure, then
$\RR(X)$ is isotropic, see \cite{DPS-duke}. The definition of
separability and strong isotropicity of $\RR(V,K)$ is inspired by the much
studied case of hyperplane arrangements \cites{CSc-adv, SS-tams},
and it is one of the points of this paper that reveal the relevance of these
conditions for resonance varieties studied in different geometric contexts.

With this terminology in place, we can state one of our main  results:

\begin{theorem}
\label{thm:sep->reduced}
Let $K\subseteq \bigwedge^2 V$ be a linear subspace, and suppose
all irreducible components of $\RR(V,K)$ are linear subspaces of $V^{\vee}$.
\begin{enumerate}[itemsep=2pt,topsep=-1pt]
\item \label{t1-1}
If $W(V,K)$ is separable, then the projectivized resonance scheme
$\Rproj(V,K)$ is reduced and its components are disjoint.
\item \label{t1-2}
If $\Rproj(V,K)$ is reduced and isotropic, then it is separable.
\end{enumerate}
\end{theorem}

Assume now that the resonance $\RR(V, K)$ is linear and denote by
$\overline{V}_1^{\vee},\dots , \overline{V}_k^{\vee}$ the (linear) components of $\RR(V,K)$.
For each $1\le t\le k$, the inclusion $\overline{V}_t^{\vee} \subseteq V^{\vee}$
corresponds to a linear projection, $\pi_t\colon V\onto \overline{V}_t$. Setting
$\overline{K}_t \coloneqq \big(\bwedge^2\pi_t \big)(K)\subseteq \bwedge^2 \overline{V}_t$,
we obtain in this way Koszul modules $W(\overline{V}_t, \overline{K}_t)$.
As an application of Theorem \ref{thm:sep->reduced}, we prove the following
result, which describes the dimensions of the graded pieces of a separable Koszul module.

\begin{theorem}
\label{thm:separable}
Suppose $W(V,K)$ is a separable Koszul module. Then
\[
\dim\, W_q(V,K) = \sum_{t=1}^k \dim\, W_q(\overline{V}_t, \overline{K}_t),
\]
for all $q\gg 0$.
\end{theorem}

If the resonance is strongly isotropic, then all the subspaces
$\overline{K}_t\subseteq \bigwedge^2 \overline{V}_t$ appearing in the
statement of Theorem \ref{thm:separable} are trivial. For a vector space
$U$ we have the canonical identification
$W_q(U, 0)\cong \ker\bigl\{U\otimes \Sym^q U\rightarrow \Sym^{q+1}U\bigr\}$,
and so $\dim  W_q(U,0)=\binom{q+\dim U}{q+2}$. Therefore, we obtain in this
case a simple combinatorial asymptotic formula for the Hilbert series of $W(V,K)$.

\begin{corollary}
\label{cor:strongly-separable}
Suppose $W(V,K)$ is a strongly isotropic Koszul module and let us write
$\RR(V, K)=\overline{V}_1^{\vee} \cup \cdots \cup \overline{V}_k^{\vee}$.
Then for all $q\gg 0$
\[
\dim\, W_q(V,K) = \sum_{t=1}^k (q+1)\binom{q+\dim \overline{V}_t}{q+2}.
\]
\end{corollary}

We mention that formulae similar to the one above are obtained in \cite{AFRSS} in
the monomial case, using methods specific to square-free monomial ideals, even in
the absence of separability or isotropicity. The intersection point between \cite{AFRSS}
and Theorem \ref{thm:separable} is represented by Proposition \ref{prop:iso-sep-graph},
which gives necessary and sufficient conditions in terms of the associated graph for
the resonance of a monomial subspace $K$ to be isotropic or separable.

\subsection{Chen ranks of groups}
\label{intro:chen ranks}

One of the main applications of this theory is to the computation of the
Chen ranks of large classes of finitely generated groups. Given such a
group $G$, we denote by $G=\Gamma_1(G)\supseteq \cdots \supseteq
\Gamma_q(G)\supseteq \Gamma_{q+1}(G) \supseteq \cdots$ its lower
central series, defined by $\Gamma_{q+1}(G)=[\Gamma_q(G), G]$.
This is a normal, central series; the direct sum of its successive quotients,
\[
\gr(G)\coloneqq \bigoplus_{q\geq 1} \Gamma_q(G)/\Gamma_{q+1}(G).
\]
acquires the structure of a graded Lie algebra, generated by its degree $1$
piece, $G/G'$, where $G'=[G,G]$.  Let $G/G''$ be the maximal metabelian
quotient of $G$, where $G''=[G',G']$. The {\em Chen ranks}\/ of $G$ are defined
\cites{CSc-adv, PS-imrn, SS-tams} as the graded ranks
of this (finitely generated) Lie algebra, that is,
\begin{equation}
\label{eq:chenranks}
\theta_q(G)\coloneqq \rank \gr_q (G/G'')=
\rank \Gamma_{q}\bigl(G/G''\bigr)/\Gamma_{q+1}\bigl(G/G''\bigr).
\end{equation}

From the cohomology algebra of $G$ in low degrees, one can extract the
Koszul module
\[
W(G)\coloneqq W(V,K)
\]
of the group $G$, by setting $V=H_1(G,\k)$ and letting
$K^{\perp}$ be the kernel of the  cup-product map
$\cup_G\colon \bigwedge^2 H^1(G,\k) \to H^2(G,\k)$.
We let $\RR(G)=\RR(V,K)$ be the corresponding resonance
variety of the group and $\Rproj(G)=\Rproj(V,K)$ the corresponding
projectivized resonance scheme.

As shown e.g. in \cites{PS-imrn, SW-forum}, if the group $G$ is $1$-formal
(that is, its pronilpotent Lie algebra admits a quadratic presentation), then
\begin{equation}
\label{eq:theta-w}
\theta_q(G)=\dim W_{q-2} (G)
\end{equation}
for all $q\ge 2$. Moreover, as shown in \cite{DPS-duke}, all the components
of $\RR(G)$ are (rationally defined) linear subspaces of $V^{\vee}$.
As an immediate application of Corollary \ref{cor:strongly-separable}, we
recover the main result  of Cohen and Schenck \cite[Theorem A]{CSc-adv},
in a somewhat stronger form.

\begin{corollary}
\label{cor:cs}
Let $G$ be a $1$-formal group, and assume $\mathcal{R}(G)$ is strongly isotropic.
Denoting by $\overline{V}_1^{\vee},\dots , \overline{V}_k^{\vee}$  the (linear)
components of $\RR(G)$, we have
\[
\theta_q(G) = \sum_{t=1}^k (q-1)\binom{q+\dim \overline{V}_t-2}{ q}
\]
for all $q\gg 0$.
\end{corollary}

Compared to \cite{CSc-adv}, the assumption that the components of $\mathcal{R}(G)$
be projectively disjoint is no longer needed; indeed, in view of Theorem \ref{thm:sep->reduced},
that property follows automatically from separability. Note also that, thanks to \cite{SW-forum},
we dropped the assumption that $G$ admit a commutator-relators finite presentation.

\subsection{Resonance for vector bundles on curves}
\label{intro:res-vb}
As explained in detail in \cite{AFRW}, see also \cite[Theorem 1.7]{FJP},
a major source of resonance varieties is provided by vector bundles in
algebraic geometry. To a vector bundle $E$ on a complex algebraic variety $X$,
one can associate a Koszul module $W(X,E)\coloneqq W(V,K)$ and a resonance
variety $\RR(X,E)\coloneqq \RR(V,K)$, by setting $V\coloneqq H^0(X,E)^{\vee}$ and
\[
\begin{tikzcd}[column sep=18pt]
K^{\perp}\coloneqq \ker \Bigl\{d_2 \colon \bwedge^2 H^0(X,E)\ar[r]&
H^0\bigl(X,\bwedge^2 E\bigr)\Bigr\},
\end{tikzcd}
\]
where $d_2$ is the natural map defined as follows. To any element $s\wedge s'\in\bwedge^2 H^0(X,E)$,
one associates the section $d_2(s\wedge s')\in H^0\bigl(X,\bwedge^2 E\bigr)$ whose value
at any point $x$ is precisely the vector $s(x)\wedge s'(x)$ in the fibre $E(x)$ of $E$ over $x$.
In the rank-$2$ case, $d_2$ is the determinant map and is sometimes denoted in literature
by $\det$. It was observed in \cite{AFRW} that $\RR(X,E)$ parametrizes \emph{subpencils}\/
of the vector bundle $E$, that is, line subbundles $A\hookrightarrow E$ such that
$h^0(X,A)\geq 2$. It is thus natural to seek to characterize geometrically those vector
bundles $E$ for which the resonance is strongly isotropic.

For simplicity, we assume $X$ is a smooth algebraic curve and $E$ is a rank $2$ stable
vector bundle of degree $d$ on $X$. Following Drinfeld and Laumon \cite{La}, we say
that $E$ is \emph{very stable}\/ if it has no non-zero nilpotent Higgs fields, that is, the
space $H^0\bigl(X, \omega_X\otimes \End(E)\bigr)$ contains no non-zero
nilpotent elements.  It has been proven in \cite{PPN} that $E$ is very stable if and
only if the space $H^0\bigl(X,\omega_X\otimes \End(E)\bigr)$ is closed inside
the moduli space $\mathrm{Higgs}_X(2,d)$ of rank $2$ Higgs fields on $X$.
For further connections between very stability and mirror symmetry, see \cite{HH}.
It turns out that this concept is closely related to the strong isotropicity of $\RR(X,E)$.

Given a line bundle $L$ on $X$, let $\SU_X(2,L)$ be the moduli space of semistable rank
$2$ vector bundles $E$ on $X$ with $\bwedge^2 E\cong L$. The locus of stable but not very stable
bundles on $X$ is known to be a divisor in $\SU_X(2,L)$, see \cite{PP}. For a stable vector
bundle $E\in \SU_X(2,L)$ and a  positive integer $a$, we let
\[
W^1_a(E)\coloneqq \Bigl\{A\in \Pic^a(X): h^0(X,A)\geq 2 \mbox{ and } A\hookrightarrow E\Bigr\}
\]
be the variety of degree $a$ subpencils of $E$. The following result describes
the structure of resonance varieties of rank $2$ vector bundles on curves
in terms of linear systems $\abs{A}$:

\begin{theorem}
\label{thm:sepvb2}
Fix a general curve $X$ of genus $g$ and a line bundle $L\in \Pic^d(X)$,
where $2g+2\leq d\leq 3g+1$. For a general stable rank $2$ vector bundle
$E$ on $X$ with $\bwedge^2 E\cong L$, the following hold:
\begin{enumerate}[itemsep=2.5pt,topsep=-1pt]
\item \label{higgs1}
If $d<3g+1$, then $\RR(X,E)=\{0\}$.

\item \label{higgs2} If $d=3g+1$, then $\RR(X,E)$ is strongly isotropic and moreover
\[
\Rproj(X,E)=\bigcup_{a=\lceil\frac{g+2}{2}\rceil}^{g+1}\Rproj_a(X,E),
\]
where $\Rproj_a(X,E)=\bigcup\bigl\{\abs{A}: A\in W^1_a(E)\bigr\}$
is a disjoint union of
\[
\# W^1_a(E)=\frac{2^{2a-g-2}}{g+1} \binom{g+1}{g-a+1, g-a+2, 2a-g-2}
\]
disjoint projective lines.
\end{enumerate}
\end{theorem}

For a general stable bundle $E$ of degree $3g+1$ with fixed determinant as in
part \eqref{higgs2}, observe that $\mu(\omega_X\otimes E^{\vee})\leq g-1$,
therefore $h^1(X,E)=0$ and then by Riemann--Roch we have that
$h^0(X,E)=g+3$. The strong isotropicity of $\RR(X,E)$ implies that the
intersection $\P\ker(d_2)\cap \Grass_2\bigl(H^0(X,E)\bigr)$ consists of
\[
\deg \Grass_2(g+3)=\frac{(2g+2)!}{(g+1)!\cdot (g+2)!}=
\sum_{a=\lceil\frac{g+2}{2}\rceil}^{g+1}
\frac{2^{2a-g-2}}{g+1}\binom{g+1}{g-a+1, g-a+2,2a-g-2}
\]
reduced points. Furthermore, when $\deg(E)>3g+1$, the resonance
$\RR(X,E)$ is no longer linear, see Remark \ref{rmk:nonruled}. The reason in
the statement of Theorem \ref{thm:sepvb2} we restrict to the case $d\geq 2g+2$,
is that when $d<2g+2$, then $h^0(X,E)\leq 3$ and the statement becomes trivial.
In the particular case $a=g+1$ we have a more precise result that holds for every
very stable vector bundle, rather than for a general one.

\begin{corollary}
\label{cor:Vafa}
Let $X$ be a general curve of genus $g$ and let $E$ be a very stable vector bundle
of degree $3g+1$. Then $\Rproj_{g+1}(X,E)$ consists of $2^g$ disjoint projective lines.
\end{corollary}

Note that the number of maximal line subbundles of a vector bundle has been the
subject of study in enumerative geometry \cites{LN, Oxb} and recently in the context
of Tevelev degrees \cite{FL}.

\subsection{Resonance of Kodaira fibrations}
\label{intro:kodaira}

A Kodaira fibration is a submersion $f\colon X\rightarrow B$ from a smooth algebraic
surface to a smooth projective curve $B$ of genus $b\geq 2$, such that all fibres of $f$
are smooth curves of genus $g$ varying in moduli. Equivalently, $f$ corresponds to a
morphism $B\rightarrow M_g$ into the moduli stack of smooth genus $g$ curves.
Kodaira fibrations are objects of intense study in algebraic geometry \cites{CR, Kod},
in the theory of surface bundles \cites{CFT, Kas, ChenLei, Salter}, and in geometric group
theory \cite{LIPy}. It is an open question posed independently by Catanese and Salter
whether there are algebraic surfaces that admit more than two Kodaira fibration structures.
Our  Theorem \ref{thm:separable} turns out to be useful in this context.
We have the following result describing the resonance of double Kodaira fibrations
(for the case of surfaces admitting a unique Kodaira fibration structure,
see Lemma \ref{lemma:complint}).

\begin{theorem}
\label{thm:koddouble}
Let $X$ be a compact algebraic surface which admits two independent Kodaira
fibrations, $\Sigma_{g_1}\hookrightarrow X\xrightarrow{f_1} B_1$ and
$\Sigma_{g_2}\hookrightarrow X\xrightarrow{f_2} B_2$.
Assume that the product map $f=(f_1,f_2)\colon X\rightarrow B_1\times B_2$
induces an isomorphism on $H^1(-,\C)$ and a monomorphism on $H^2(-,\C)$.
Then
\begin{enumerate}[itemsep=3pt, topsep=-1pt]
\item
The resonance scheme $\RR(X)=f_1^*H^1(B_1,\C)\cup f_2^*H^1(B_2,\C)$
is separable.
\item
The Chen ranks of $\pi_1(X)$ are given by the following formula, for $q\gg 0$,
\[
\theta_q\bigl(\pi_1(X)\bigr)=(q-1)\left(\binom{2b_1+q-2}{q}+
\binom{2b_2+q-2}{q}\right)-\binom{2b_1+q-3}{q-2}-\binom{2b_2+q-3}{q-2}.
\]
\end{enumerate}
\end{theorem}

We explain in Section \ref{subsec:a-k} how the hypothesis of Theorem \ref{thm:koddouble}
are verified for the Atiyah--Kodaira surfaces constructed in \cites{At, Kod}; consequently,
the theorem provides a formula for the Chen ranks (in large degrees) of these surfaces.
Finally, we also conjecture in \S\ref{subsec:a-k} that the projective resonance $\Rproj(G)$
of any K\"ahler group $G$ is reduced and that a ``Chen ranks formula" in the spirit
of Theorem \ref{thm:separable} and generalizing the one in Theorem \ref{thm:koddouble}
always holds.

In this paper we refrain from studying the Koszul modules and the Chen ranks of
hyperplane arrangements. They will form the subject of the forthcoming paper \cite{AFRS-effective}.

\begin{ack}
We thank the referee for a careful reading of the paper and for valuable comments and
suggestions that have led to improvements in the exposition.

{\small{
Aprodu  was supported by the Romanian Ministry of Research and Innovation,
CNCS -- UEFISCDI, grant PN-III-P4-ID-PCE-2020-0029, within PNCDI III.
and by the PNRR grant CF 44/14.11.2022 {\em Cohomological Hall algebras
of smooth surfaces and applications}.
Farkas was supported by the DFG Grant \emph{Syzygien und Moduli} and by the
ERC Advanced Grant SYZYGY.
Raicu was supported by the NSF Grants No.~1901886 and~2302341.
Suciu was supported in part by Simons Foundation Collaboration Grants for
Mathematicians \#354156 and \#693825.
This project has received funding from the European Research Council (ERC)
under the European Union Horizon 2020 research and innovation program
(grant agreement No. 834172).
}}
\end{ack}

\section{Koszul modules and resonance schemes}
\label{sec:defs}

\subsection{Koszul modules}
\label{subsec:Koszul}

We now recast some of the notions introduced in the previous section in
a more convenient setting, following the approach adopted in \cite{PS-crelle}
and developed in \cites{AFPRW,AFPRW2, AFRW}.

Once again, let $V$ be a finite-dimensional vector space over
an algebraically closed field $\k$ of characteristic $0$,
and let $K\subseteq\bwedge^2 V$ be a subspace.  We also let
$K^{\perp}\subseteq \bigwedge^2 V^{\vee}$ be the
subspace of all linear functionals
vanishing on $K$. Unless otherwise specified, all tensor products
will be over $\k$.

Set  $S\coloneqq \Sym(V)$.
Upon picking a basis for $V$, this algebra may be identified with the
polynomial ring $\k[x_1,\dots,x_n]$, where $n=\dim V$.   Let
$\bigl(\bigwedge^{\bullet} V \otimes S, \delta\bigr)$ be the
corresponding Koszul complex. As already mentioned, the \defi{Koszul module}\/ $W(V,K)$
is the middle homology of the chain complex
\begin{equation}
\label{eq:def-WVK}
\begin{tikzcd}[column sep=22pt]
K \oo S \ar[rr, "\left. \delta_2\right|_{K \oo S}"] && V\oo S(1) \ar[r, "\delta_1"] & S(2).
\end{tikzcd}
\end{equation}

It is readily
seen that $W(V,K)$ is the cokernel of the $S$-linear map
\begin{equation}
\label{eq:presentation}
\begin{tikzcd}[column sep=18pt]
\bigl(\bwedge^3V\oplus K\bigr)\otimes  S \ar[r]&  \bwedge^{2} V\otimes S
\end{tikzcd}
\end{equation}
with matrix
$\delta_3+\iota \otimes \id_S$, where $\iota\colon K\inj \bwedge^2 V$ is
the inclusion map. One has an alternate presentation of $W(V,K)$ as a
cokernel of a morphism of graded $S$-modules:
\begin{equation}
\label{eq:presentation2}
\begin{tikzcd}[column sep=18pt]
\bwedge^3V \otimes S(-1) \ar[r]&  \Bigl(\bwedge^{2} V/K\Bigr)\otimes S
\ar[r]& W(V,K) \ar[r]& 0
\end{tikzcd}
\end{equation}

A $\k$-linear map $\varphi\colon V\to \ol{V}$  induces a linear map
$\bwedge^2 \varphi\colon \bwedge^2 V\to \bwedge^2 \ol{V}$
and we let $\ol{K}$ be the image of $K$ under this map.
Writing $\ol{S}\coloneqq \Sym(\ol{V})$, the map
$\varphi$ induces a ring morphism
$\pi\colon S\to \ol{S}$.  Denoting by $\widetilde{W}(\ol{V},\ol{K})$
the cokernel of the composed map
\[
\begin{tikzcd}[column sep=18pt]
(\bwedge^3V\oplus K)\otimes  S \ar[r]&  \bwedge^{2} V\otimes S
\ar[r]& \bwedge^{2} \ol{V}\otimes S
\end{tikzcd}\] we obtain a morphism
of graded $S$-modules (which by abuse of notation
we also call $\varphi$),
\begin{equation}
\label{eq:phi-W}
\begin{tikzcd}[column sep=20pt]
\varphi \colon W(V,K) \ar[r, two heads]& \widetilde{W}(\ol{V},\ol{K}) .
\end{tikzcd}
\end{equation}
Clearly, $W(\ol{V},\ol{K})=\widetilde{W}(\ol{V},\ol{K})\otimes _S\ol{S}$.
Moreover, if $\varphi\colon V\to \ol{V}$ is surjective,
then the morphism \eqref{eq:phi-W} is also surjective.

\subsection{Koszul modules and differentials}
\label{subsec:WVK-diffs}
Next, we explain the relationship between Koszul modules
and the sheaf $\Omega=\Omega^1_{\P}$ of differential forms
on the projective space $\P\coloneqq \P(V^{\vee})$.
Consider the Euler sequence
\begin{equation}
\label{eq:Euler}
\begin{tikzcd}[column sep=16pt]
0 \ar[r]& \Omega  \ar[r]& V \oo \mc{O}_{\P}(-1)  \ar[r]& \mc{O}_{\P}  \ar[r]& 0\, .
\end{tikzcd}
\end{equation}

If $p=[f]\in\P$, where $0\neq f\in V^{\vee}$, then the restriction of \eqref{eq:Euler}
to $p$ identifies with
\begin{equation*}
\label{eq:Om-on-fiber}
\begin{tikzcd}[column sep=16pt]
 0 \ar[r]& \ker(f) \ar[r]& V \ar[r, "f"]& \k \ar[r]& 0 \, .
\end{tikzcd}
\end{equation*}

Using the left exactness of the global sections functor, and the fact that
$W(V,0) = \ker(\delta_1)$ (as follows from \eqref{eq:def-WVK}), we obtain from
\eqref{eq:Euler} that
\begin{equation}
\label{eq:wqv0}
W_q(V,0) = H^0\bigl(\P,\Omega(q+2)\bigr).
\end{equation}
for all $q\in\bb{Z}$. Using \eqref{eq:def-WVK}, we can also write
\begin{equation}
\label{eq:WVK-quot-WV0}
\begin{tikzcd}[column sep=16pt]
 W_q(V,K) = \coker\Bigl\{ K \oo \Sym^q V \ar[r]& W_q(V,0)\Bigr\}.
 \end{tikzcd}
 \end{equation}

The graded $S$-module $W(V,K)$ induces a coherent sheaf $\W(V,K)$ on
the projective space $\P= \P\bigl(V^{\vee}\bigr)$. We refer to $\W(V,K)$,
as the \defi{Koszul sheaf} associated with the pair $(V,K)$.

From the natural surjection $\bwedge^2 V \oo \mc{O}_{\P}(-2) \onto \Omega$,
coupled with the inclusion $K\subseteq\bwedge^2 V$, we obtain a map
$K\oo \mc{O}_{\P}(-2) \to \Omega$, which, upon twisting, taking global
sections, and using \eqref{eq:WVK-quot-WV0} yields the identification
\begin{equation*}
\label{eq:wqvk-h0}
\begin{tikzcd}[column sep=16pt]
W_q(V,K) = \coker \Bigl\{
H^0\bigl(\P,K\oo \mc{O}_{\P}(q)\bigr) \ar[r]& H^0\bigl(\P,\Omega(q+2)\bigr) \Bigr\},
\end{tikzcd}
\end{equation*}
for all $q\in\Z$. Accordingly, the associated Koszul sheaf can be realized as
\begin{equation}
\label{eq:wqvk-op}
\begin{tikzcd}[column sep=16pt]
\W(V,K) = \coker\Bigl\{K\oo \mc{O}_{\P} \ar[r]& \Omega(2)\Bigr\},
\end{tikzcd}
\end{equation}
and in particular we have the identification $\W(V,0)=\Omega(2)$
of sheaves on $\P$.

If $i\colon \ol{V}^{\vee}\subseteq V^{\vee}$ is a linear subspace,
let $\pi\colon V\surj \ol{V}$ be the dual map.
Set $\ol{K}\coloneqq \bwedge^2 \pi(K)$. By the discussion above, we have
a surjective morphism of graded $S$-modules,
\begin{equation}
\label{eq:onto-pi-W}
\begin{tikzcd}[column sep=20pt]
\pi \colon W(V,K) \ar[r, two heads]& \widetilde{W}(\ol{V},\ol{K}) \, .
\end{tikzcd}
\end{equation}
The morphism $\pi$ from \eqref{eq:onto-pi-W} corresponds
to a surjective morphism of coherent $\mc{O}_{\P}$-sheaves,
\begin{equation}
\label{eq:onto-pi-W-sheaf}
\begin{tikzcd}[column sep=20pt]
\pi \colon \W(V,K) \ar[r, two heads]& i_* \W(\ol{V},\ol{K}) \, .
\end{tikzcd}
\end{equation}

\subsection{The presentation of the Koszul sheaf}
\label{subsec:stalks-koszul}
Let $\ol{\P} = \mathbf{P}\bigl(\ol{V}^{\vee}\bigr) \hookrightarrow
\P=\mathbf{P}\bigl(V^{\vee}\bigr)$ be the inclusion map and let $p$
be a point in $\ol{\P}$.
Write $p=[e_1]$ for $0\neq e_1\in \ol{V}^{\vee}$ and complete $\{e_1\}$ to
bases of $\ol{V}^{\vee}$ and $V^{\vee}$.  We let $X_1,\ldots,X_n$
denote the homogeneous coordinates on $\P$ corresponding to the choice of
basis for $V^{\vee}$, and we let $x_i \coloneqq X_i/X_1$ denote the local coordinates at
the point $p$. Letting $A=\mc{O}_{\P,p}$ denote the local ring of $\P$ at $p$,
we have
\begin{equation}
\label{eq:akk}
A = \k[x_2,\ldots,x_n]_{\mf m},
\end{equation}
where $\mf{m} = (x_2,\ldots,x_n)$. Using \eqref{eq:Euler}, we have that the
stalk of $\Omega$ at $p$ can be described as
\begin{equation}
\label{eq:omegap}
\Omega_p = \bigoplus_{j=2}^n A\cdot dv_j,
\end{equation}
where $dv_j \coloneqq v_j - x_j\cdot v_1$. Here we identify $v_j$ and
$v_j\otimes \frac{1}{X_1}\in V\otimes \mathcal{O}_{\P,p}(-1)$.

\begin{proposition}
\label{prop:stalk-wp}
When viewed as an $A$-module, the stalk of the Koszul sheaf
$\W=\mc{W}(V,K)$ at the point $p\in\ol{\P}$ has presentation
\begin{equation}
\label{eq:k-otimes-a}
\begin{tikzcd}[column sep=22pt]
K \oo A \ar[r, "\pd"]& \Omega_p \ar[r, "\nu", two heads]& \W_p \, ,
\end{tikzcd}
\end{equation}
where $\pd$ is the restriction to $K\oo A$ of the $A$-linear map
$\delta_2\colon \bwedge^2 V\oo A \to \Omega_p$ by
\begin{gather}
\begin{aligned}
\label{eq:def-pd}
\delta_2(v_s\wedge v_t)
&= x_s\cdot dv_t - x_t\cdot dv_s
&&\text{for $2\le s,t\le n$},\\
\delta_2(v_1\wedge v_t)
&=dv_t
&&\text{for $2\le t \le n$}.
\end{aligned}
\end{gather}
\end{proposition}

\begin{proof}
The Koszul differential $\delta_2\colon \bwedge^2 V \otimes S\to V\otimes S$
yields a map $\delta_2\colon \bwedge^2 V \otimes A\to V\otimes A$, which in
turn restricts to a map $\pd \colon K\otimes A\to \Omega_p$. We then obtain
the commuting diagram
\begin{equation}
\label{eq:pd-delta}
\begin{tikzcd}[column sep=22pt]
K \otimes A \ar[r, "\pd"] \ar[d, hook]
& \Omega_p \ar[r, "\nu", two heads] \ar[d, equal] & \W_p \\
\bwedge^2V \otimes A  \ar[r, "\delta_2"]& \Omega_p \ar[ur]
\end{tikzcd}
\end{equation}
and it is now readily verified that the map $\pd$ has the desired form.
\end{proof}

By abuse of notation, we will write $dv_t$ for the element of $\W_p$
which is the image under $\nu$ of the corresponding element $dv_t\in\Omega_p$.

\subsection{Resonance varieties and schemes}
\label{subsec:resvars}
It has been proven in \cite{PS-crelle} that the set-theoretic
support of the $S$-module $W(V,K)$ is given by the \defi{resonance},
\[
\RR(V,K)\coloneqq\Bigl\{a\in V^\vee  :  \text{ there exists $b\in
V^\vee$ such that $a\wedge b\in K^\perp\setminus \{0\}$} \Bigr\}\cup \{0\}.
\]
where $K^{\perp}\subseteq \bigwedge^2 V^{\vee}$. In other words, for
a given $K\subseteq \bigwedge^2 V$ we have that $W_q(V,K)=0$ for
$q\gg 0$ if and only if $\RR(V,K)=0$.  The annihilator of the Koszul module
$W(V,K)$, which we denote by $I(V,K)$, is a homogeneous ideal in $S$.
We let
\begin{equation}
\label{eq:def-res-affine-scheme}
\RR(V,K) \coloneqq \op{Spec}\bigl(S/I(V,K)\bigr)
\end{equation}
denote the \defi{affine resonance scheme},
which is the scheme-theoretic support of $W(V,K)$ inside $V^{\vee}$.
Since this is the only scheme structure we will use, there is no ambiguity
in using the same notation for the scheme $\RR(V,K)$ and its underlying
variety.

\begin{remark}
\label{rem:fitting}
There is another possible scheme structure
on $\RR(V,K)$, given by the
Fitting ideal $\Fitt_0 W(V,K)$. However, the scheme structure
given by the annihilator $I(V,K)$ is the minimal one, and is invariant
under closed embeddings of ambient affine spaces. This property of the
annihilator support will be used in the proof of Theorem~\ref{thm:sep->disjoint}. For a comparison of these two scheme structures in the case of right-angled Artin groups, see Example \ref{ex:4paths}.
\end{remark}

The scheme-theoretic support of the Koszul sheaf $\W(V,K)$ defined
in \eqref{eq:wqvk-op} is called the \defi{projective resonance scheme},
and is denoted by
\begin{equation}\label{eq:rproj-def}
\Rproj(V,K) \coloneqq \op{Proj}\bigl(S/I(V,K)\bigr) \, .
\end{equation}
Using the same convention as above, we will denote the underlying
projective resonance variety by $\Rproj(V,K)$, although our main
interest is in the scheme structure.

We isolate  here and in the later sections certain desirable properties
of the resonance schemes and their underlying varieties that arose from
the study of resonance varieties of complements of hyperplane arrangements.

\begin{definition}
\label{def:res-lin}
 We say that the resonance of $(V,K)$ is \defi{linear} if $\mc{R}(V,K)$ is
a union of linear subspaces of $V^{\vee}$.  Furthermore, we say that $\mc{R}(V,K)$ is
\defi{projectively disjoint} if the irreducible components of $\Rproj(V,K)$
are pairwise disjoint. Finally, we say that the resonance is
\defi{projectively reduced} if $\Rproj(V,K)$ is a reduced scheme.
\end{definition}

\begin{remark}
\label{rem:lines}
If $\Rproj(V,K)$ is a finite union of lines, then these lines are
necessarily projectively disjoint.
\end{remark}

\subsection{Projective geometry interpretation of resonance}
\label{subsec:proj-geom}
The projectivized resonance has a simple description in terms of
projective geometry, using the incidence variety of the Grassmannian.
Let $\Gr\coloneqq \Grass_2(V^\vee)$ be the
Grassmannian of $2$-planes in $V^\vee$, viewed as a
subset of $\mathbf{P}\big(\!\bigwedge^2V^\vee\big)$
via the Pl\"{u}cker embedding. Consider the diagram
\begin{equation}
\label{eq:proj-reson}
\begin{tikzcd}[column sep=24pt]
 \mathbf{P}\times \Gr
& \Xi\ar[r, "\pr_2"] \ar[d, "\pr_1"] \ar[l, hook']
&\Gr \ar[r, hook]
&\mathbf{P}\big(\!\bigwedge^2V^\vee\big) ,\\
 &\mathbf{P} &&
\end{tikzcd}
\end{equation}
where $\Xi=\bigl\{(p,L)\in \mathbf{P} \times \Gr : p\in L\bigr\} $
is the incidence variety and $\pr_1$ and $\pr_2$ are the two
projections as in \eqref{eq:proj-reson}. The
natural bijection between the points of $\Gr$ and the set of lines in
$\mathbf{P}$ is given by the correspondence
$\pr_1\circ \pr_2^{-1}$, which maps a point
$[a\wedge b]\in  \mathbf{G}$ to the line
$L_{ab}$ in $\mathbf{P}$ passing through $[a]$ and $[b]$.
The inverse of this map is given by
$L_{ab}\mapsto\mathrm{pr}_2( \mathrm{pr_1}^{-1}([a])\cap\ \mathrm{pr_1}^{-1}([b]))$.
The next lemma readily follows.

\begin{lemma}
\label{lem:resonance-proj}
Set-theoretically, the projective resonance variety is given by
\begin{align}
\label{eq:rvk-pr}
\mathbf{R}(V,K)&=
\mathrm{pr}_1\bigl(\mathrm{pr_2}^{-1}(\Gr\cap\mathbf{P}K^\perp)\bigr).
\\
\intertext{Moreover, the following inclusion holds}
\label{eq:GcapPKperp}
\Gr\cap\mathbf{P}K^\perp&\subseteq
\mathrm{pr}_2\bigl(\mathrm{pr_1}^{-1}(\mathbf{R}(V,K))\bigr).
\end{align}
\end{lemma}

In the particular case of hyperplane arrangements, the equality \eqref{eq:rvk-pr}
was previously established by Lima-Filho and Schenck in
\cite[Proposition 2.1]{Filho-Schenck}.  As illustrated in
Example \ref{ex:isotropic-res} below, the inclusion \eqref{eq:GcapPKperp}
is not an equality in general.

Since the projection $\mathrm{pr}_2$
realizes $\Xi$ as the projectivization of the universal rank-two bundle
on the Grassmannian, the projectivized resonance is
covered by lines. More precisely, if $[a]\in \mathbf{R}(V,K)$
and $0\ne a\wedge b\in K^\perp$, then the line $L_{ab}$ joining $[a]$ and $[b]$
is included in $\mathbf{R}(V,K)$.
Note that a point in the resonance may be contained in a
higher-dimensional linear subspace of $\mathbf{R}(V,K)$.
Specifically, if $[a]\in \mathbf{R}(V,K)$, then
\begin{equation}
\label{eq:ha}
H_a\coloneqq \bigl\{[b]\in \mathbf{P}: a\wedge b\in K^\perp\bigr\}
\end{equation}
is a linear subspace of $\mathbf{P}$ of dimension at least one
and completely contained in $\mathbf{R}(V,K)$.

The projective geometry description of resonance yields plenty of examples
of projective varieties that are not resonance varieties.

\begin{example}
\label{ex:generic-points}
Even among varieties
covered by lines, there are simple examples that are not resonance varieties.
Let  $\Gamma=\{p_1,\ldots,p_m\}\subseteq \mathbf{G}$ be
a set of points that are not contained in any hyperplane of
$\mathbf{P}(\bigwedge^2V^\vee)$. Then the set
$X\coloneqq \pr_1\circ \pr_2^{-1}(\Gamma)$
is the union of $m$ disjoint lines, $X=L_1\sqcup\cdots\sqcup L_m$.
Suppose $X=\mathbf{R}(V,K)$, for some subspace $K\subseteq \bigwedge^2V$.
By the above discussion, for any $[a]\in \mathbf{R}(V,K)$, the linear subspace
$H_a$ is contained in $\mathbf{R}(V,K)$. It follows that for any $j$ and any
$[a],[b]\in L_j$ we have $a\wedge b\in K^\perp$, which implies that
$\Gamma=\mathbf{G}\cap \mathbf{P}K^\perp$, which is a contradiction.
Hence, $X$ is not a resonance variety.
\end{example}

\begin{remark}
\label{rem:cjl}
The above situation is in stark contrast with what happens for the closely related
characteristic varieties. For instance, given any integrally
defined hypersurface $Y\subseteq (\C^*)^n$, there exists a finitely presented group $G$
with $H_1(G,\Z)=\Z^n$ for which the characteristic variety $\mathcal{V}(G)\coloneqq
\{\rho\in \Hom(G,\C^*) : H_1(G,\C_{\rho}) \ne 0\}$ is isomorphic to $Y\cup \{1\}$; see
\cite[Lemma 10.3]{SYZ}.
\end{remark}

\section{Isotropic and separable subspaces of resonance}
\label{sect:isotropic}

Recall that  $V$ is a finite-dimensional vector space and
$K\subseteq \bwedge^2 V$ is a linear subspace.

\begin{definition}
\label{def:isotropic}
A linear subspace $\ol{V}^{\vee}\subseteq V^{\vee}$ is said to be \defi{isotropic}
(with respect to $K$) if $\bwedge^2 \ol{V}^{\vee}\subseteq K^{\perp}$.
\end{definition}

The isotropicity property can be described by passing
to the quotient.  If $\pi\colon V\surj \ol{V}$ is the corresponding projection, recalling  that $K^{\perp}$ is the kernel of the  projection
$\bwedge^2 V^{\vee} \surj K^{\vee}$, setting
 $\ol{K}:=\bwedge^2(\pi)(K)$ we observe that   $\ol{V}^{\vee}$ is isotropic if and only if
$\ol{K}=0$.

\begin{definition}
\label{def:isotropic-res}
We say that the resonance variety $\mathcal{R}(V,K)$ is \defi{isotropic}
if it is linear and each of its irreducible components is isotropic.
\end{definition}

By definition, any isotropic subspace $\ol{V}^{\vee}\subseteq V^{\vee}$
is automatically contained in the resonance variety $\mathcal{R}(V,K)$.
Moreover, $\Rproj(V,K)$ is a union of isotropic lines; more precisely,
\begin{equation}
\label{eq:rproj-decomp}
\Rproj(V,K) =\bigcup_{a,b\in V^{\vee}: a\wedge b\in K^{\perp}} L_{ab}.
\end{equation}

\begin{example}
\label{ex:isotropic-res0}
Assume $V^{\vee}=\langle e_1, \ldots, e_n\rangle$ and set
$K^\perp = \langle e_1\wedge e_2\rangle\subseteq \bigwedge^2 V^{\vee}$.
Then clearly $\mathcal{R}(V,K)=\langle e_1, e_2\rangle$, which is isotropic.
\end{example}

On the other hand, a linear component of the resonance variety is not
necessarily isotropic, as shown by the following example.

\begin{example}
\label{ex:isotropic-res}
Consider a subspace $K\subseteq\bwedge^2 V$ of dimension $m$,
where $1\le m\le n-2$ and $n=\dim V$. Since the sheaf  $\Omega^1_{\P}(2)$
has rank $n-1$, formula \eqref{eq:wqvk-op} implies that
$\mc{R}(V,K)$ coincides with  $V^{\vee}$.
However, $\mc{R}(V,K)$ is not isotropic,
as $K^{\perp}$ is a proper subspace of $ \bigwedge^2 V^{\vee}$.
\end{example}

\begin{remark}
Another instance when the resonance coincides with the ambient space is
when $V$ decomposes as a non-trivial direct sum of $\k$-vector subspaces,
$V=U_1\oplus U_2$, so that $U_1^{\vee}\wedge U_2^{\vee}\subseteq K^{\perp}$.
Then it follows that $\mc{R}(V,K)=V^{\vee}$, see \cite[Lemma 5.2]{PS-mathann}.
\end{remark}

For the next definition,  let $E \coloneqq \bw V^{\vee}$
be the exterior algebra on the dual vector space $V^{\vee}$, and write
$\langle U \rangle_E$ for the ideal in $E$ generated by a subset $U\subseteq E$.

\begin{definition}
\label{def:separable} A linear space  $\ol{V}^{\vee}\subseteq V^{\vee}$
is \defi{separable} (with respect to $K\subseteq \bwedge^2 V$) if
\begin{equation}
\label{eq:separable-subspace}
K^{\perp} \cap \langle \ol{V}^{\vee}\rangle_E \subseteq
\bwedge^2 \ol{V}^{\vee}.
\end{equation}
We  say that $\ol{V}^{\vee}$ is \defi{strongly isotropic}
if it is separable and isotropic, that is,
\[
K^{\perp} \cap \langle \ol{V}^{\vee}\rangle_E =
\bwedge^2 \ol{V}^{\vee}.
\]

We say that the resonance variety $\mathcal{R}(V,K)$ is \defi{separable}
if  it is linear and each of its irreducible components is separable.
Likewise, the resonance is \defi{strongly isotropic} if it is isotropic
and separable.
\end{definition}

\begin{example}
\label{ex:extremes}
If $K^\perp=0$, then any subspace is separable with respect to $K$. At the
other end of the spectrum, if $K^\perp=\bwedge^2 V^{\vee}$, then the only separable
subspace is the trivial one.
\end{example}

The next example will be used in the proof of Theorem \ref{thm:sep->disjoint}.

\begin{example}
\label{ex:zero-res}
Suppose $\mc{R}(V,K)=V^{\vee}$. Then the resonance variety
is separable. Moreover, the scheme structure is
clearly reduced. If $\mc{R}(V,K)=\{0\}$, then the resonance variety
is automatically separable. The scheme structure is not necessarily reduced;
however, the projectivized resonance is empty, and thus reduced.
\end{example}

\begin{example}
\label{ex:non-sep-res}
An example of non-separable resonance is obtained
if $V^\vee=\langle e_1,e_2,e_3,e_4\rangle$ and
$K^\perp=\langle e_1\wedge e_2, e_1\wedge e_3+ e_2\wedge e_4\rangle$.
Then $\mathcal{R}(V,K)=\langle e_1, e_2\rangle$, which is not separable.
\end{example}

As shown in the next example, separable subspaces are not
necessarily contained in the resonance.

\begin{example}
\label{ex:separable-res}
Take again $V^\vee=\langle e_1,e_2,e_3,e_4\rangle$ and
$K^\perp = \langle e_1 \wedge e_2 + e_3 \wedge e_4\rangle$.
Then the subspace $\ol{V}^{\vee} = \langle e_1,e_2\rangle$ is separable, yet
$\mc{R}(V,K) = \{0\}$.
\end{example}

\subsection{A local view of separability}
\label{subsec:alt-sep}
If  $n=\dim V$ and $\ol{n}=\dim \ol{V}$,
fix a basis $(e_1,\ldots,e_n)$ of $V^{\vee}$ such that
$(e_1,\ldots,e_{\ol{n}})$ is a basis for $\ol{V}^{\vee}$.
Set $U\coloneqq\ker \big\{\pi\colon V\to\ol{V}\big\}$.
Letting $(v_1,\ldots,v_n)$ denote the dual basis of $V$, we obtain a
direct-sum decomposition,
\begin{equation}
\label{eq:decomp-bw2V}
\bwedge^2\, V = L \oplus M \oplus H,
\end{equation}
where
\begin{align}
\label{eq:decomp-lmod}
L &= \bigl\langle v_s \wedge v_t : s,t\leq \ol{n}\bigr \rangle \notag\cong \bwedge^2\, \ol{V},\\
M &= \bigl \langle v_s \wedge v_t : s\leq \ol{n} \mbox{ and }  t>\ol{n}\bigr \rangle \cong \ol{V}\otimes U,\\
H &= \bigl \langle v_s \wedge v_t : s,t>\ol{n}\bigr \rangle \notag\cong \bwedge^2\, U.
\end{align}

Observe that
\begin{equation}
\label{eq:mvee}
M^\vee=\big(\bwedge^2\,
V^\vee\cap\langle \ol{V}^{\vee}\rangle_E \big)/\bwedge^2\ol{V}^\vee.
\end{equation}
Consider now the map $\bwedge^2 \pi\colon \bwedge^2 V\to \bwedge^2\ol{V}$. Then
$\ker\big(\bwedge^2\pi\big)^\vee=\bwedge^2V^\vee /\bwedge^2\ol{V}^\vee$,
and hence  $M^\vee\subseteq \ker\big(\bwedge^2\pi\big)^\vee$,
inducing a surjection $\ol{\pi}\colon \ker\!\big(\bwedge^2\pi\big)\surj  M$.
Let
\begin{equation}
\label{eq:projection-m}
\begin{tikzcd}[column sep=18pt]
p_M\colon K\cap \ker\!\big(\bwedge^2\pi\big)\ar[r]&  M
\end{tikzcd}
\end{equation}
be the restriction of $\ol{\pi}$ to the subspace $K\cap \ker\!\big(\bwedge^2\pi\big)$.
The next lemma provides a convenient local criterion for verifying the
separability of $\ol{V}^\vee$, that we will often use for concrete applications.

\begin{lemma}
\label{lem:alt-sep}
With notation as above,
\begin{enumerate}[itemsep=2pt,topsep=-1pt]
\item \label{vee1}
$\coker(p_M)\cong \left((K^\perp\cap \langle \ol{V}^{\vee}\rangle_E)/
(K^\perp\cap \bigwedge^2\ol{V}^\vee\right)^{\vee}$.
\item \label{vee2}
The subspace $\ol{V}^\vee\subset V^\vee$ is separable if and only if
the map $p_M$ is surjective.
\end{enumerate}
\end{lemma}

\begin{proof}
The exact sequence
\begin{equation*}
\label{eq:kv-exact}
\begin{tikzcd}[column sep=19pt]
0\ar[r]& K\cap\ker\big(\bwedge^2\pi\big) \ar[r]& K \ar[r]& \bwedge^2\ol{V}
\end{tikzcd}
\end{equation*}
gives rise by dualizing to the exact sequence
\begin{equation}
\label{eq:kv-dual-exact}
\begin{tikzcd}[column sep=19pt]
K^\perp\cap \bwedge^2 \ol{V}^{\vee} \ar[r]&
\bwedge^2\ol{V}^\vee \ar[r]&  K^\vee \ar[r]&
(K\cap\ker(\bwedge^2\pi))^\vee \ar[r] & 0,
\end{tikzcd}
\end{equation}
from which we infer that  the kernel of the map
$K^\vee\to (K\cap\ker(\bigwedge^2\pi))^\vee$ is equal to
$\bigwedge^2\ol{V}^\vee /(K^\perp\cap \bigwedge^2\ol{V}^\vee)$.
Consider now the diagram
\begin{equation}
\label{eq:alt-sep}
\begin{tikzcd}[cramped, sep=small]
  & 0 \arrow[d]
  & 0 \arrow[d]
  & 0 \arrow[d]
  \\
0 \arrow[r]
  & K^\perp\cap \bigwedge^2\ol{V}^\vee \arrow[r]\arrow[d]
  & \bigwedge^2\ol{V}^\vee \arrow[r]\arrow[d]
  & \bigwedge^2\ol{V}^\vee /(K^\perp\cap \bigwedge^2\ol{V}^\vee) \arrow[d]
  \arrow[r] & 0
  \\
0 \arrow[r]
  & K^\perp\cap \langle \ol{V}^{\vee}\rangle_E  \arrow[r]\arrow[d]
  & \bigwedge^2 V^\vee\cap \langle \ol{V}^{\vee}\rangle_E  \arrow[r]\arrow[d]
  & K^\vee \arrow[d]
  \\
0 \arrow[r]
  & (K^\perp\cap \langle \ol{V}^{\vee}\rangle_E)/(K^\perp\cap \bigwedge^2\ol{V}^\vee)  \arrow[r]\arrow[d]
  & M^\vee \arrow[r, "p_M^\vee"]\arrow[d]
  & \big(K\cap\ker(\bigwedge^2\pi)\big)^\vee \arrow[d]
  \\
  & 0
  & 0
  & 0
\end{tikzcd}.
\end{equation}

By \eqref{eq:mvee}, the
middle column is exact; the left column and the top row are clearly exact;
finally, the middle row is also exact.
Thus, the bottom row is also exact and the diagram commutes, by the Snake Lemma.
Dualizing the bottom row in diagram \eqref{eq:alt-sep} yields the first claim.
The second claim now follows from claim \eqref{vee1} and Definition \ref{def:separable}.
\end{proof}

When $\ol{V}^\vee$ is isotropic with respect to
$K$, the above separability criterion simplifies.
First observe that $\ol{V}^\vee$ is isotropic if and only if
$K\subseteq  \ker(\bigwedge^2\pi)$. We then have the following:

\begin{corollary}
\label{cor:isotropic-separable}
An isotropic subspace $\ol{V}^\vee\subseteq V^{\vee}$ is separable if and only if
the map $p_M\colon K\to M$ given by (\ref{eq:projection-m})  is surjective.
\end{corollary}

We refer to Lemma \ref{lem:separable-multiplication} below for a dual version
of Corollary \ref{cor:isotropic-separable}. With the help of diagram \eqref{eq:proj-reson},
we now obtain the following projective-geometric interpretation of separability.

\begin{lemma}
\label{lem:proj-separable}
Let $\ol{V}^\vee \subseteq V^\vee$ be a linear subspace and denote by
$\ol{\mathbf{P}}\subseteq\mathbf{P}$ the associated projective subspace.
Then $\ol{V}^\vee$ is separable if and only if
\begin{equation}
\label{eq:separable-proj-subspace}
\Spn(\pr_2\pr_1^{-1}(\ol{\mathbf{P}}))\cap\mathbf{P}(K^\perp)
\subseteq \mathbf{P}\big(\bwedge^2\ol{V}^\vee\big).
\end{equation}
\end{lemma}

\begin{proof}
Note that the projective subspace $\Spn(\pr_2\pr_1^{-1}(\ol{\mathbf{P}}))
\subseteq \mathbf{P}\big(\bwedge^2V^\vee\big)
$
spanned by $\pr_2\pr_1^{-1}(\ol{\mathbf{P}})$ is
the projectivization of the space of quadrics in the exterior ideal
$\langle \ol{V}^\vee\rangle_E$. The claim follows.
\end{proof}

Lemma \ref{lem:alt-sep} provides useful information regarding the restriction
of the Koszul sheaf over the projectivization $\ol{\mathbf{P}}\coloneqq \mathbf{P}(\ol{V})$
of a separable subspace $\ol{V}$. Recall that $\ol{K}$ denotes the
image of $K$ under the projection $\bwedge^2 V\surj \bwedge^2 \ol{V}$
and that the conormal bundle $N_{\ol{\mathbf{P}}/\mathbf{P}}^{\vee}$ is isomorphic to
$\ker(\pi)\otimes\mathcal{O}_\mathbf{P}(-1)$.
We infer that the morphism
$\mathcal{W}(V,K)|_\mathbf{\ol{P}} \to\mathcal{W}(\ol{V},\ol{K})$
is an isomorphism provided
$(K\cap\ker(\bigwedge^2\pi))\otimes\mathcal{O}_\mathbf{\ol{P}}
\to \ker(\pi)\otimes\mathcal{O}_\mathbf{\ol{P}}(1) $ is surjective.
Given that $M$ is naturally identified with $\ker(\pi)\otimes \ol{V}$,
 Lemma \ref{lem:alt-sep} ensures that if $\ol{V}^\vee$ is separable, this morphism is
 already surjective on global sections, and hence the
 isomorphism $\mathcal{W}(V,K)|_\mathbf{\ol{P}} \isom \mathcal{W}(\ol{V},\ol{K})$
 is guaranteed. In the next section we shall prove a stronger statement, namely, that the map
 $\mathcal{W}(V,K) \to\mathcal{W}(\ol{V},\ol{K})$ is a
 local isomorphism at any point of $\ol{\mathbf{P}}$.

\subsection{Strong isotropy via the multiplication map}
\label{subsec:iso-mult}
We now recast Corollary \ref{cor:isotropic-separable} in a setting that will
prove to be useful when studying Koszul modules associated to vector bundles
on varieties. Assume $\ol{V}^\vee\subseteq V^\vee$ is an isotropic subspace,
put $U^\vee=V^\vee/\ol{V}^\vee$ with projection map $\rho\colon V^\vee\to U^\vee$,
and denote by $\phi\colon \bigwedge^2 V^\vee\to K^\vee$ the dual to the inclusion
map $K\inj \bwedge^2 V$.  By the isotropy hypothesis, the map $\phi$
induces a \emph{multiplication map},
\begin{equation*}
\label{eq:mult-map}
\begin{tikzcd}[column sep=19pt]
\mu\colon \ol{V}^\vee\otimes U^\vee\ar[r]& K^\vee,
\end{tikzcd}
\end{equation*}
by setting $\mu\bigl(\alpha\otimes \rho(\beta)\bigr)\coloneqq \phi(\alpha\wedge \beta)$.
Note that absent the isotropy condition, $\mu$ is not well-defined. With the notation
from the previous sections, under the natural identification $M\cong \ol{V}\otimes U$,
the map $\mu$ is dual to $p_M$. The following result which can be deduced
from Lemma \ref{lem:alt-sep}. For  the convenience of the reader, we also give here
a self-contained proof.

\begin{lemma}
\label{lem:separable-multiplication}
The subspace $\ol{V}^\vee$ is strongly isotropic if and only if  $\mu$ is injective.
\end{lemma}

\begin{proof}
Assume first that $\mu$ is injective. Let
$\omega=\sum\alpha_i\wedge \beta_i\in K^\perp\cap\langle \ol{V}\rangle_E$,
with $\alpha_i\in \ol{V}^\vee$ linearly independent. Set $b_i\coloneqq \rho(\beta_i)$.
Since $\phi(\omega)=0$, it follows that $\mu(\sum\alpha_i\otimes b_i)=0$,
which implies $\sum\alpha_i\otimes b_i=0$. Since $\alpha_i$ are linearly
independent, the elements $b_i$ all vanish, hence $\beta_i\in \ol{V}^\vee$
for all $i$. In particular, $\omega\in\bigwedge^2\ol{V}^\vee$.

Conversely, if $\ol{V}^\vee$ is strongly isotropic, choose $w=\sum\alpha_i\otimes b_i\in\ker(\mu)$
with $\alpha_i\in\ol{V}^\vee$ linearly independent. Lifting each $b_i$ to $\beta_i\in V^\vee$,
we obtain an element $\omega=\sum\alpha_i\wedge\beta_i\in\ker(\phi)=K^\perp$.
Since $\omega$ belongs also to $\langle \ol{V}^\vee\rangle$ and since $\ol{V}^\vee$ is
isotropic, it follows that $\omega\in\bigwedge^2 \ol{V}^\vee$, and hence
$\omega=\sum\alpha'_j\wedge\beta'_j$ with $\alpha'_j,\beta'_j\in \ol{V}^\vee$.
Since $\rho(\beta'_j)=0$, we have that $w=0$.
\end{proof}

This form of Corollary \ref{cor:isotropic-separable} is particularly
useful in the case of vector bundles. We refer to Section \ref{subsec:bundles}
for a more detailed discussion.

\section{Koszul modules with separable resonance}
\label{sect:separate}
Our goal in this section is to study those Koszul modules
whose resonance varieties are separable, and to prove
Theorem~\ref{thm:sep->reduced}, part \eqref{t1-1}
and Theorem~\ref{thm:separable} from the Introduction.

\subsection{Bases in $K$ with respect to separable subspaces}
\label{subsec:sep-bases}
As usual, let $V$ be a finite-dimensional vector space and let
$K\subseteq \bwedge^2 V$ be a subspace.
The next lemma provides explicit bases in $K$, related to a given
separable linear subspace  $\ol{V}^{\vee}\subseteq V^{\vee}$.
Let $\bwedge^2\, V = L \oplus M \oplus H$ be the direct-sum
decomposition given by \eqref{eq:decomp-lmod}.

\begin{lemma}
\label{lem:2nd-basis}
Let $\ol{V}^{\vee}$ be a separable subspace of $V^{\vee}$ with respect
to $K$.  There exists then a basis of $K$ of the form
$\{\a_{s,t}:s\leq \ol{n}, t>\ol{n}\} \cup \{\b_1, \ldots, \b_N\}$,
where $N$ is a non-negative integer, such that, for each $s\leq \ol{n}$, $t>\ol{n}$,
and $1\leq j\leq N$, we have
 \begin{equation}
 \label{eq:abs}
 \a_{s,t} = v_s \wedge v_t + h_{s,t},\quad \b_j =\ell_j + h_j,
 \end{equation}
 for some collection of elements $\ell_j \in L$ and $h_{s,t}, h_j\in H$.
 \end{lemma}

\begin{proof}
By Lemma \ref{lem:alt-sep}, the map
$p_M$ is surjective.  Hence, we can lift each $v_s\wedge v_t\in M$ to
an element $\a_{s,t}\in K\cap (M\oplus H)$. We write
$\alpha_{s,t}=v_s\wedge v_t+h_{s,t}$, where $h_{s,t}\in H$.
If we take $(\b_1,\ldots,\b_N)$ to be a basis of an algebraic complement of
$K\cap (M\oplus H)$ in $K$, the elements $\b_j$
have the desired form, namely, $\b_j=\ell_j+h_j$, for some $\ell_j\in L$
and $h_j\in H$.
\end{proof}

Note that for some of the elements $\beta_j$ in the above lemma (e.g., those
contained in the kernel of $p_M$ from \eqref{eq:projection-m}), the corresponding
element $\ell_j$ is zero.

\subsection{Koszul sheaves and separable components}
\label{subsec:K-sheaves}
Consider the map of Koszul modules
$\pi \colon W(V,K)\to W(\ol{V}, \ol{K})$ from \eqref{eq:onto-pi-W}
and the corresponding map of Koszul sheaves,
\begin{equation}
\label{eq:uppi-W}
\begin{tikzcd}[column sep=20pt]
\uppi\colon \mc{W}(V,K) \ar[r, two heads]& \mc{W}(\ol{V},\ol{K}).
\end{tikzcd}
\end{equation}
We write $\mc{W} = \mc{W}(V,K)$ and
$\ol{\mc{W}} = \mc{W}(\ol{V},\ol{K})$, and we let
$\ol{\P} = \mathbf{P}\bigl(\ol{V}^{\vee}\bigr) \hookrightarrow
\P=\mathbf{P}\bigl(V^{\vee}\bigr)$
be the inclusion map. Fix a point $p\in \P$.
Write $p=[e_1]$ for $0\ne e_1\in \ol{V}^{\vee}$ and complete $\{e_1\}$ to
bases of $\ol{V}^{\vee}$ and $V^{\vee}$ as before.
By Proposition \ref{prop:stalk-wp}, the stalk of $\W$ at $p$ has presentation
$K \oo A \xrightarrow{\,\pd\,} \Omega_p \xrightarrow{\,\nu\,} \W_p \to 0$,
where $\pd$ is the restriction to $K\oo A$ of the map defined on $
\bwedge^2 V\oo A$ by \eqref{eq:def-pd}.

Assume now that $\ol{V}^{\vee}$ is a separable subspace of
$V^{\vee}$ with respect to $K$, and consider the basis of $K$
obtained in Lemma~\ref{lem:2nd-basis}.

\begin{lemma}
\label{lem:dvt=0}
For $t=\ol{n}+1,\ldots,n$, we have that $dv_t=0$ in $\W_p$.
\end{lemma}

\begin{proof}
We let $\mc{M}\subseteq\W_p$ denote the $A$-submodule generated by
$dv_{\ol{n}+1},\ldots,dv_n$. To prove that $\mc{M}=0$, it suffices by Nakayama's
lemma to show that $\mc{M} \subseteq \mf{m}\cdot\mc{M}$. We note that
by \eqref{eq:decomp-bw2V} and \eqref{eq:def-pd} we have that
\begin{equation}
\label{eq:nupd-h=0}
\nu(\pd(h)) \in \mf{m}\cdot\mc{M}\:\mbox{ for all $h\in H$}.
\end{equation}
Since
\[
0 = \nu(\pd(\a_{1,t})) = \nu(\delta_2(\alpha_{1,t }))=\nu(dv_t)+ \nu(\delta_2(h_{1,t})),
\]
it follows that $dv_t\in \mf{m}\cdot\mc{M}$ for $t=\ol{n}+1,\ldots,n$,
showing that $\mc{M} \subseteq \mf{m}\cdot\mc{M}$ and concluding the proof.
\end{proof}

Combining Lemma~\ref{lem:dvt=0} with the following result will allow us
to conclude that locally at $p$, the Koszul sheaf $\W$ is scheme-theoretically
supported on $\ol{\P}=\mathbf{P}\bigl(\ol{V}^{\vee}\bigr)$.

\begin{lemma}
\label{lem:xtdvs=0}
We have that $x_t\cdot dv_s = 0$ in $\W_p$, for all $2\leq s\leq \ol{n}$ and $t>\ol{n}$.
\end{lemma}

\begin{proof}
As in the proof of Lemma~\ref{lem:dvt=0}, we have that
\begin{equation}
\label{eq:nupd-ast}
0= \nu(\pd(\a_{s,t}))  =\nu( x_s\cdot dv_t - x_t\cdot dv_s) + \nu(\delta_2(h_{s,t})).
\end{equation}
Using \eqref{eq:nupd-h=0} together with Lemma~\ref{lem:dvt=0} we
conclude that $\nu(x_t\cdot dv_s)=0$, as desired.
\end{proof}

\begin{proposition}
\label{prop:uppi-local-isom}
Assume $\ol{V}^{\vee}$ is a separable linear subspace contained in $\RR(V,K)$.
Then the map $\uppi\colon \mc{W}(V,K) \surj \mc{W}(\ol{V},\ol{K})$
induces an isomorphism on stalks at each point $p\in\ol{\P}$.
\end{proposition}

\begin{proof}
If $B\coloneqq  A/(x_{\ol{n}+1},\ldots,x_n)$ denotes the local ring $\mc{O}_{\ol{\P},p}$
of $\ol{\P}$ at $p$, then Lemmas~\ref{lem:dvt=0} and~\ref{lem:xtdvs=0} show
that the map $\nu\colon \Omega_p \surj \W_p$ factors
through $\ol{\Omega}_p$, where $\ol{\Omega}=\Omega_{\ol{\P}}$.
We thus obtain an alternative presentation of $\W_p$, this time over
$B$, given by
\begin{equation}
\label{eq:kperp-sum}
\begin{tikzcd}[column sep=20pt]
K \oo B \ar[r, "\pd"] & \ol{\Omega}_p \ar[r, two heads] & \W_p \, .
\end{tikzcd}
\end{equation}
Observe that $\pd(v_s\wedge v_t) = 0$ if $t>\ol{n}$, and therefore the map $\pd$
factors through $\ol{K} \oo B$. Since $\ol{\W}_p$ has presentation
\begin{equation}
\label{eq:kperp-sum-bis}
\begin{tikzcd}[column sep=20pt]
\ol{K}\oo B \ar[r]& \ol{\Omega}_p \ar[r, two heads] & \ol{\W}_p\, ,
\end{tikzcd}
\end{equation}
we conclude that the natural map $\mc{W}_p \to \ol{\W}_p$ is an isomorphism,
thereby completing the proof of Proposition~\ref{prop:uppi-local-isom}.
\end{proof}

\subsection{Separability and reduced scheme structure}
\label{subsec:sep-koszul}
Next, we apply Proposition~\ref{prop:uppi-local-isom} to obtain
a characterization of the separable irreducible components of the
resonance scheme.

\begin{theorem}
\label{thm:sep->disjoint}
Each separable irreducible component of $\Rproj(V,K)$ is a reduced, isolated
component of projectivized resonance.
\end{theorem}

\begin{proof}
Let $P=\mathbf{P}\bigl(\ol{V}^{\vee}\bigr)$ be a separable component
of $\Rproj(V,K)$ (possibly non-reduced). Suppose that
$P$ intersects some other component $Q$ (not necessarily linear), and
let $p\in P\cap Q$ be any point. Using Proposition~\ref{prop:uppi-local-isom},
we infer that locally at $p$ the resonance scheme $\Rproj(V,K)$ is a closed
subscheme of $\ol{\P}=P_{\rm{red}}$,
which is a contradiction, since $P_{\rm{red}}$ does not contain $Q$.
Therefore, $P$ is an isolated irreducible component, i.e., a connected
component of $\Rproj(V,K)$.

To prove that the scheme $P$ is reduced, consider any point $p\in P$. Applying
Proposition~\ref{prop:uppi-local-isom} as in the previous paragraph,
we deduce that locally at $p$, the component $P$ is contained in
$\ol{\P}=P_{\rm{red}}$, and so $P = P_{\rm{red}}$.
\end{proof}

A consequence of this theorem proves
(in a stronger form) Theorem~\ref{thm:sep->reduced}, part \eqref{t1-1}.

\begin{corollary}
\label{cor:sep-red}
If the whole resonance is separable, then it is also projectively disjoint
and reduced.
\end{corollary}

As we shall show in Example \ref{ex:proj reduced non separable},
the converse of Theorem \ref{thm:sep->disjoint} is not necessarily true.

\subsection{Decomposition of separable Koszul modules}
\label{subsec:decomp-sep}

Suppose $\mc{R}(V,K)$ is linear, with components
$\overline{V}_1^{\vee},\dots ,\overline{V}_k^{\vee}$.
For each $1\le t\le k$, the inclusion $\overline{V}_t^{\vee} \subseteq V^{\vee}$
corresponds to a linear projection, $\pi_t\colon V\onto \overline{V}_t$.
We set $\overline{K}_t \coloneqq \big(\bwedge^2\pi_t \big)(K)$,
and obtain in this way Koszul modules $W(\overline{V}_t, \overline{K}_t)$,
together with natural surjective maps
$\pi_t\colon W(V,K) \onto \widetilde{W}(\overline{V}_t, \overline{K}_t)$
as in \eqref{eq:onto-pi-W}, as well as corresponding maps of Koszul sheaves,
$\uppi_t\colon \mc{W}\onto \mc{W}_t \equalscolon \mc{W}(\overline{V}_t, \overline{K}_t)$.

The next theorem proves Theorem~\ref{thm:separable} from the Introduction.

\begin{theorem}
\label{thm:sepKoszul-iso}
Suppose $W(V,K)$ is a separable Koszul module. Then the morphism
\begin{equation}
\label{eq:Pi map}
\begin{tikzcd}[column sep=18pt]
\Pi\coloneqq ( \pi_1,\dots ,\pi_k) \colon
W(V,K) \ar[r]& \bigoplus_{t=1}^k \widetilde{W}(\overline{V}_t, \overline{K}_t),
\end{tikzcd}
\end{equation}
is an isomorphism in sufficiently large degrees.
\end{theorem}

\begin{proof}
Consider the map of sheaves
\begin{equation}
\label{eq:ww map}
\begin{tikzcd}[column sep=18pt]
\boldsymbol{\Pi} \coloneqq ( \uppi_1,\dots ,\uppi_k) \colon
\W  \ar[r]& \bigoplus_{t=1}^k \W_t .
\end{tikzcd}
\end{equation}

Let $p\in\Rproj(V,K)$ a point; then $p$ belongs to, say, $\overline{V}_t^{\vee}$.
Since $\overline{V}_t^{\vee}$ is a separable subspace, Proposition \ref{prop:uppi-local-isom}
implies that $\uppi_t$ is an isomorphism at $p$. Moreover, outside $\Rproj(V,K)$,
the map $\uppi_t$ is also an isomorphism, since the stalks on each side are zero.
Corollary \ref{cor:sep-red}
now shows that the map $\boldsymbol{\Pi}$  is an isomorphism.
Since $\boldsymbol{\Pi}$ is the sheafification of $\Pi$, it follows
that $\Pi$ is an isomorphism in degrees $q\gg 0$.
\end{proof}

\section{Isotropic components of the resonance}
\label{sec:isotropic}

As mentioned in Section \ref{subsec:sep-koszul}, the projective
resonance can be reduced in the absence of separability. We will
show in this section that if the resonance is isotropic, then the two
conditions---projectively reduced and, respectively,
separable---are actually equivalent.

Suppose  $\ol{V}^{\vee} \subseteq V^{\vee}$ is a linear, irreducible component
of $\mc{R}(V,K)$ which is isotropic, and let~$P$ denote the corresponding linear
component of the resonance scheme $\Rproj(V,K)$. The reduced
subscheme structure on $P$ is given by $P_{\rm{red}}=\ol{\P}$, where
$\ol{\P} = \mathbf{P}\bigl(\ol{V}^{\vee}\bigr)$. The goal of this section is to analyze
the condition that $P$ is reduced. We prove the following theorem.

\begin{theorem}
\label{thm:iso-reduced}
Suppose that $\ol{V}$ is an isotropic component of $\mc{R}(V,K)$,
and let $P$ denote the corresponding component of $\Rproj(V,K)$.
The following are equivalent:
\begin{enumerate}
\item\label{it:all-p} $P$ is reduced.
\item\label{it:one-p} $P$ is generically reduced.
\item\label{it:str-iso} $\ol{V}$ is strongly isotropic.
\end{enumerate}
\end{theorem}

\begin{proof}
It is clear that $\eqref{it:all-p} \Rightarrow \eqref{it:one-p}$. Since strongly
isotropic components are separable, it follows from Theorem~\ref{thm:sep->disjoint}
that the implication $\eqref{it:str-iso} \Rightarrow \eqref{it:all-p}$ also holds.
The only implication left to prove is therefore $\eqref{it:one-p} \Rightarrow \eqref{it:str-iso}$.
Let $\pi\colon V \to \ol{V}$ be the associated surjective homomorphism.
Recall
that $\ol{V}^{\vee}$ is isotropic if and only if $(\bwedge^2\pi)(K) = 0$. We then
obtain as in \eqref{eq:onto-pi-W} a surjective homomorphism of Koszul modules,
$\pi\colon W(V,K) \onto W(\ol{V},0)$, and a corresponding surjection at the level of
sheaves, $\uppi\colon \mc{W}(V,K) \onto \mc{W}(\ol{V},0) = \ol{\Omega}(2)$,
where $\ol{\Omega}=\Omega_{\ol{\P}}$ is viewed as a
sheaf on $\P$ via the closed immersion $\ol{\P}\hookrightarrow\P$.

Suppose now that $p=[e_1]$ is a reduced point of $P$, where $0\neq e_1\in\ol{V}^{\vee}$.
Upon choosing bases for $\ol{V}^{\vee}$ and $V^{\vee}$ as in Section~\ref{sect:separate},
consider the decomposition $\bwedge^2 V = L \oplus M \oplus H$ given by
\eqref{eq:decomp-bw2V} and \eqref{eq:decomp-lmod},
and let $p_M\colon K\to M$ denote the restriction to $K$
of the second-component  projection $\bwedge^2 V\to  M$.
As shown in Corollary \ref{cor:isotropic-separable},
condition \eqref{it:str-iso} is equivalent to $p_M$ being surjective;
the remainder of the proof will focus on establishing this surjectivity.
Since the reduced locus of $P$ is non-empty, it is dense in $P$, and so it is not
contained in any other irreducible component of $\Rproj(V,K)$. We will therefore
assume that the point $p$ is chosen to lie only on the component $P$ of
$\Rproj(V,K)$. It follows that $\mc{W}(V,K)$ is supported on $\ol{\P}$ locally
at $p$. We prove the following claim, which is the key technical point of our proof.

\begin{claim*}
\label{lem:uppi-loc-iso}
The map $\uppi\colon \mc{W}(V,K) \onto \ol{\Omega}(2)$ is a local isomorphism
at the point $p$.
\end{claim*}

\begin{proof}[Proof of Claim]
Since $\mc{W}(V,K)$ is scheme-theoretically supported on $\ol{\P}$ at $p$,
and since $\ol{\Omega}$ is a locally free sheaf on $\ol{\P}$, we can think of
$\uppi$ locally at $p$ as a split surjection of sheaves on $\ol{\P}$.
To conclude, it is then enough to
check that it is an isomorphism on the fiber at $p$.

Recall from \eqref{eq:wqvk-op} that
$\W(V,K) =\coker\bigl\{ K\oo \mc{O}_{\P} \lra \Omega(2)\bigr\}$.  We need to
prove the exactness on the fiber at $p$ of the exact sequence
\begin{equation}
\label{eq:W-to-Omega}
\begin{tikzcd}[column sep=20pt]
K\oo \mc{O}_{\P} \ar[r]& \Omega(2) \ar[r]& \ol{\Omega}(2) \ar[r]& 0 \, .
\end{tikzcd}
\end{equation}
Exactness on the right follows because
$\ol{\P}\hookrightarrow\P$ is a closed immersion. For the next argument, which
is independent on the earlier choice of bases, we write $f=e_1\in V^{\vee}$, so that $p=[f]$.
Recalling that $\ol{V}^{\vee}\subseteq V^{\vee}$,
there exists a linear form
$\ol{f}\colon \ol{V}\to\k$ such that $f=\ol{f}\circ \pi$.
The restriction of \eqref{eq:W-to-Omega} yields a complex,
\begin{equation*}
\label{eq:fvee-kvee}
\begin{tikzcd}[column sep=19pt]
K  \ar[r]& \ker(f) \ar[r]& \ker(\ol{f}),
\end{tikzcd}
\end{equation*}
where the second map is induced by $\pi$, and the first one is induced
by the Koszul differential $a\wedge b \mapsto f(a)\cdot b-f(b)\cdot a$.
By duality, its exactness reduces to that of
\begin{equation}
\label{eq:3term-Kvee}
\begin{tikzcd}[column sep=19pt]
 \ker(\ol{f})^{\vee}   \ar[r]& \ker(f)^{\vee}  \ar[r]& K^{\vee},
\end{tikzcd}
\end{equation}
where  the
second is the map $\psi$ in the commutative diagram:
\begin{equation}
\label{eq:triangle}
\begin{tikzcd}[column sep=30pt,row sep=28pt]
V^{\vee} \ar[r,"\wedge f"] \ar[dr, two heads] &
\bwedge^2 V^{\vee} \ar[r, two heads] & K^{\vee}  .\\
& \ker(f)^{\vee} \ar[u, "\wedge f"'] \ar{ur}[swap]{\psi}& &
\end{tikzcd}
\end{equation}
The exactness of \eqref{eq:3term-Kvee} amounts to the following statement:
if $g\in V^{\vee}$ has the property that $g\wedge f\in K^{\perp}$, then the
restriction of $g$ to $\ker(f)$ is obtained via composition with $\pi$ from
a linear form on $\ker(\ol{f})$.  This follows if we can prove
that $g\in\ol{V}^{\vee}$. In view of \eqref{eq:def-WVK}, the condition
that $g\wedge f\in K^{\perp}$ implies
that $\bigl\langle g,f \bigr\rangle \subseteq \mc{R}(V,K)$, which in turn
implies that the line $L_{f,g}$ through $p=[f]$
and $[g]$ lies in $\Rproj(V,K)$. Since the only component of $\Rproj(V,K)$
containing $p$ is $P$, it follows that $L_{f,g}\subseteq P$.  In
particular, $g\in\ol{V}^{\vee}$,  completing the proof of the claim.
\end{proof}

Having established the main claim, we proceed with the proof of the
surjectivity of $p_M$. We let $A=\mc{O}_{\P,p}$ as in \eqref{eq:akk}.
Using the presentation of $\W_p$ from \eqref{eq:k-otimes-a}, the
description \eqref{eq:omegap} of the stalk at $p$ of $\Omega$,
and the analogous description of $\ol{\Omega}_p$, we derive
from the main claim the existence of an exact sequence
\begin{equation}
\label{eq:local-res1}
\begin{tikzcd}[column sep=20pt]
K \oo A \ar[r, "\pd"]& \bigoplus_{j=2}^n A \cdot dv_j \ar[r]&
 \bigoplus_{j=2}^{\ol{n}} B \cdot d\ol{v}_j \ar[r]& 0 \, ,
\end{tikzcd}
\end{equation}
where $B = A/(x_{\ol{n}+1},\dots,x_n)\cong\mc{O}_{\ol{\P},p}$,
$\ol{v}_j = \pi(v_j)\in\ol{V}$, and $d\ol{v}_j = \ol{v}_j - x_j\cdot\ol{v}_1$.
Since $\ol{V}^{\vee}$ is isotropic, it follows that $K^{\perp}\supseteq L^{\vee}$,
and therefore $K\subseteq M\oplus H$, which yields
\begin{equation}
\label{eq:kerpM}
\ker(p_M) = K \cap H.
\end{equation}

If we let $\mf{m}=(x_2,\dots,x_n)$ denote the maximal ideal of $A$, we
infer using \eqref{eq:def-pd} that $\pd$ sends $(K\cap H)\oo A$ into
$\bigoplus_{j=\ol{n}+1}^n \mf{m}\cdot dv_j$. We then obtain from \eqref{eq:local-res1}
a commutative diagram,
\begin{equation}
\label{eq:kh diagram}
\begin{tikzcd}[column sep=16pt, row sep=12pt]
(K \cap H)\oo A \ar[r] \ar[d] & \ds\bigoplus_{j=\ol{n}+1}^n
\mf{m}\cdot dv_j \ar[d] \ar[r] & 0 \ar[d] &  \\
K \oo A \ar[r] & \ds\bigoplus_{j=2}^n A \cdot dv_j \ar[r]
& \ds\bigoplus_{j=2}^{\ol{n}} B \cdot d\ol{v}_j \ar[r] & 0 \, ,
\end{tikzcd}
\end{equation}
where the vertical maps are inclusions. Using \eqref{eq:kerpM}, we
infer that $p_M(K) \cong K/(K\cap H)$.  Taking cokernels of the vertical maps
in diagram \eqref{eq:kh diagram} gives rise to an exact sequence,
\begin{equation}
\label{eq:pmk-seq}
\begin{tikzcd}[column sep=18pt]
p_M(K) \oo A \ar[r, "\partial"] & \bigg(\ds\bigoplus_{j=2}^{\ol{n}}
A \cdot dv_j\bigg) \bigoplus \bigg(\ds\bigoplus_{j=\ol{n}+1}^n
\k \cdot dv_j\bigg) \ar[r]& \ds\bigoplus_{j=2}^{\ol{n}} B \cdot d\ol{v}_j \ar[r]& 0 \, ,
\end{tikzcd}
\end{equation}
where $A/\mf{m} \cong \k$. Tensoring over $A$ with $A/\mf{m}^2$ preserves
right exactness; therefore, since
$\k\oo_A A/\mf{m}^2 = \k$ and $B\oo_A A/\mf{m}^2 = B/\mf{m}^2$,
we obtain a further exact sequence,
{\small
\begin{equation}
\label{eq:local-res2}
\begin{tikzcd}[column sep=18pt]
p_M(K) \oo A/\mf{m}^2 \ar[r, "\partial"]
& \bigg(\ds\bigoplus_{j=2}^{\ol{n}} A/\mf{m}^2 \cdot dv_j\bigg) \bigoplus
\bigg(\ds\bigoplus_{j=\ol{n}+1}^n \k \cdot dv_j\bigg)\ar[r]
& \ds\bigoplus_{j=2}^{\ol{n}} B/\mf{m}^2 \cdot d\ol{v}_j \ar[r]& 0\, .
\end{tikzcd}
\end{equation}
}

It follows from \eqref{eq:def-pd} that $p_M(K) \oo (\mf{m}/\mf{m}^2)\subseteq \ker(\pd)$,
so \eqref{eq:local-res2} yields the exact sequence,
{\small
\begin{equation}
\label{eq:local-res3}
\begin{tikzcd}[column sep=18pt]
p_M(K) \oo A/\mf{m} \ar[r, "\pd"]&\bigg(\ds\bigoplus_{j=2}^{\ol{n}}
A/\mf{m}^2 \cdot dv_j\bigg) \bigoplus \bigg(\ds\bigoplus_{j=\ol{n}+1}^n \k \cdot
dv_j\bigg) \ar[r]& \ds\bigoplus_{j=2}^{\ol{n}} B/\mf{m}^2 \cdot d\ol{v}_j  \ar[r]& 0\, .
\end{tikzcd}
\end{equation}
}

Viewing this as an exact sequence of $\k$-vector spaces, and noting that
$\dim_{\k}(A/\mf{m}^2) = n$ and $\dim_{\k}(B/\mf{m}^2) = \ol{n}$, we conclude that
\begin{align}
\label{eq:dim pmk}
\dim(p_M(K)) &\geq (\ol{n}-1)\cdot n + (n-\ol{n}) - (\ol{n}-1)\cdot \ol{n} = \ol{n}\cdot(n-\ol{n}) = \dim(M).
\end{align}
Since $p_M(K)\subseteq M$, this shows that $p_M$ is
surjective, thus concluding the proof of Theorem~\ref{thm:iso-reduced}.
\end{proof}

As a consequence of this theorem, we infer that separability is equivalent to reducedness
in the isotropic case, thereby completing the proof of Theorem~\ref{thm:sep->reduced}, part \eqref{t1-2}.

\begin{corollary}
\label{cor:sep=reduced-if-iso}
 Suppose that $\mc{R}(V,K)$ is isotropic. Then $\mc{R}(V,K)$ is
 strongly isotropic if and only if $\Rproj(V,K)$ is reduced.
\end{corollary}

\begin{proof}
 If the resonance variety is strongly isotropic, then it is separable, so by
Theorem~\ref{thm:sep->reduced} part \eqref{t1-1}, we conclude that $\Rproj(V,K)$
 is reduced. The converse follows from Theorem~\ref{thm:iso-reduced}.
\end{proof}

\section{Resonance for vector bundles on curves}
\label{subsec:bundles}

\subsection{Vector bundles on projective varieties}
\label{subsec:vb-proj}
In this section we study the scheme-theoretic properties of resonance varieties
associated to vector bundles. We fix a smooth projective variety $X$,
a vector bundle $E$ on $X$ and consider the  map
\begin{equation}
\label{eq:det2}
\begin{tikzcd}[column sep=18pt]
d_2 \colon \bwedge^2 H^0(X, E)\ar[r]& H^0\bigl(X,\bwedge^2 E\bigr).
\end{tikzcd}
\end{equation}

Set $V\coloneqq H^0(X,E)^{\vee}$ and $K^{\perp}\coloneqq \ker(d_2)\subseteq\bwedge^2 V^{\vee}$.
Following \cite{AFRW}, we consider the associated Koszul module $W(X,E)\coloneqq W(V,K)$
and resonance variety $\RR(X,E)\coloneqq \RR(V,K)$. As pointed out in \cite[Proposition 4.2]{AFRW},
a section $0\neq s\in H^0(X,E)$ lies in $\RR(X,E)$ if an only it spans a \emph{subpencil}\/ of $E$, that is,
$s\in H^0(X,A)\subseteq H^0(X,E)$, where $A\hookrightarrow E$ is a line subbundle with
$h^0(X,A)\geq 2$. Thus, one is led to describe the resonance variety $\RR(X,E)$ in terms
of the variety of subpencils of $E$.

We observe that the projective resonance $\Rproj(X,E)$ of a
vector bundle enjoys a nice geometric property.

\begin{proposition}
\label{prop:max}
Let $E$ be a vector bundle over a smooth projective variety $X$. Then
\begin{enumerate}[itemsep=2pt, topsep=-1pt]
\item \label{mx1}
The projective resonance scheme $\Rproj(X,E)$ comes equipped with
a regular morphism $\chi\colon \Rproj(X,E)\rightarrow \mathcal{Z}$, where
\[
\mathcal{Z}=\Bigl\{A\hookrightarrow E: \text{$A$ is a saturated line
subbundle of $E$ with $ h^0(X,A)\geq 2$} \Bigr\}.
\]
Moreover, all fibres of $\chi$ are positive dimensional projective spaces.

\item \label{mx2}
$\Rproj(X,E)$ is linear if and only if  $\mathcal{Z}$ is finite.
\end{enumerate}
\end{proposition}

\begin{proof}
Given $[s]\in \Rproj(X,E)$, with $0\neq s\in H^0(X,E)$, let $A$ be the saturation
of the rank one subsheaf of $E$ generated by $s$. Since $[s]\in \Rproj(X,E)$,
there exists $s'\in H^0(X,E)$ which is not a scalar multiple of $s$, with
$d_2(s\wedge s')=0$. This means that $s'$ is a multiple of $s$ at the generic point,
and therefore $s'$ defines a section of $A$, in particular $h^0(X,A)\geq 2$.
We set
\[
\chi\bigl([s]\bigr)\coloneqq [A\hookrightarrow E]\in \mathcal{Z}.
\]

In order to establish \eqref{mx1}, it suffices to prove that if $\ell_1$ and $\ell_2$
are two lines in $\mathbf{P}=\mathbf{P}(V^{\vee})$ with $\ell_1\cap\ell_2=\{[s]\}$
and $\ell_1,\ell_2\subseteq \Rproj(X,E)$, then $\overline{\ell_1, \ell_2}\subseteq \Rproj(X,E)$.
Each line $\ell_i$ corresponds to a subspace
$\langle s, s_i\rangle\subseteq H^0(X, E)$ that generates a rank-one subsheaf
$A_i\hookrightarrow E$ with $h^0(X, A_i)\geq 2$. In a general fibre $E(x)$ of
$E$ the values $s(x), s_i(x)$ generate a subspace of dimension at most one.
Therefore, the space spanned by the sections $s, s_1, s_2$ generates a
rank-one subsheaf $A$ of $E$, and
$\ell_1,\ell_2\subseteq \mathbf{P}H^0(X,A)\subseteq \Rproj(X,E)$.
The sheaf $A$ being saturated, $\chi^{-1}(A)=\P H^0(X, A)\subseteq \Rproj(X,E)$.
Part \eqref{mx2}  is a direct consequence of \eqref{mx1}.
\end{proof}

\subsection{Rank $2$ vector bundles on curves}
\label{subsec:vb-curves}
We now specialize to the case when $X$ is a smooth curve of genus $g\geq 2$. For a line bundle
$L\in \Pic^d(X)$, let $\SU_X(2,L)$ be the moduli space of $S$-equivalence classes
of semistable rank $2$ vector bundles $E$ on $X$ with $\bwedge^2 E\cong L$.
It is known that $\SU_X(2,L)$ is a Fano variety of dimension $3g-3$, see \cite{Ram}.
We shall undertake a more systematic study of the geometry of the subpencils
of given degree of a rank $2$ semistable vector bundle.

For $\frac{g+2}{2}\leq a\leq g+1$, we denote by
$W^1_a(X)\coloneqq \bigl\{A\in \Pic^a(X): h^0(X,A)\geq 2\bigr\}$ the
Brill--Noether variety of pencils of degree $a$ on $X$,
see \cite[Chapter 4]{ACGH}. If $X$ is a general curve
of genus $g$, then $W^1_a(X)$ is equidimensional of
dimension $2a-g-2$ and the Petri map
\begin{equation}
\label{Petrimap}
\begin{tikzcd}[column sep=18pt]
\mu_A\colon H^0(X,A)\otimes H^0(X, \omega_X\otimes A^{\vee})\ar[r]& H^0(X,\omega_X)
\end{tikzcd}
\end{equation}
obtained by multiplication of sections is injective for each pencil $A\in W^1_a(X)$,
see \cite[Theorem 1.7]{ACGH} or \cite{Gi}.
For a vector bundle $E\in \SU_X(2,L)$ and a line subbundle $A\hookrightarrow E$
(which by twisting by $A^{\vee}$ implies that $H^0\bigl(X, E\otimes A^{\vee}\bigr)\neq 0$),
we introduce the \emph{twisted Petri map},
\begin{equation}
\label{eq:twPetri}
\beta=\beta_{E,A}\colon H^0\bigl(X,E\otimes A^{\vee}\bigr)\otimes
H^0\bigl(X, E^{\vee}\otimes \omega_X\otimes A\bigr)\longrightarrow H^0(X,\omega_X),
\end{equation}
obtained by composing the multiplication map of global sections followed by the map
induced at the level of global sections by the twist by $\omega_X$ of the trace map
$E\otimes E^{\vee}\rightarrow \mathcal{O}_X$.

We first describe in terms of extensions when the resonance $\RR(X,E)$ is strongly isotropic.

\begin{lemma}
\label{lemma:strongis}
Let $E$ be a rank $2$ vector bundle on a curve $X$ expressed as an extension
\begin{equation}
\label{eq:extrk2}
\begin{tikzcd}[column sep=18pt]
0\ar[r]& A \ar[r]& E \ar[r, "j"]& L\otimes A^{\vee}\ar[r]& 0,
\end{tikzcd}
\end{equation}
where $A$ is a pencil on $X$ with $h^0(X,A)=2$. Let $F\coloneqq \abs{A}$ be the
base locus of $A$ and let $e\in \Ext^1\bigl(L\otimes A^{\vee}, A\bigr)\cong
H^0\bigl(X, \omega_X+L-2A\bigr)^{\vee}$ be the extension class corresponding to $E$.
Then $H^0(X,A)\subseteq H^0(X,E)$ is not strongly isotropic if and only if there
exists a non-zero element $u\in H^0\bigl(X, L(-2A+F)\bigr)$ such that
\begin{equation}
\label{eq:smoothness1}
u\cdot \bigl(H^0(X,A)\otimes H^0(X, \omega_X\otimes A^{\vee})\bigr) \subseteq \ker(e).
\end{equation}
\end{lemma}
Note that in (\ref{eq:smoothness1}) the left hand side is viewed as a linear subspace of
$H^0(X, \omega_X+L-2A)$ via the multiplication map.

\begin{proof}
We apply Lemma \ref{lem:separable-multiplication}, where with the notations introduced there,
$V^{\vee}=H^0(X,E)$,
$\overline{V}^{\vee}=H^0(X,A)$,
$U^{\vee}=j\bigl(H^0(X,E)\bigr)\subseteq H^0(X, L\otimes A^{\vee})$, and
$K^{\vee}=\im(d)$, with $d$ being the determinant map given by \eqref{eq:det2}.
The map
\[
\begin{tikzcd}[column sep=18pt]
\mu\colon H^0(X,A)\otimes j\bigl(H^0(X,E)\bigr)\ar[r]& H^0(X,L)
\end{tikzcd}
\]
described in Lemma \ref{lem:separable-multiplication} can be viewed as the
restriction of the multiplication map of sections
$H^0(X,A)\otimes H^0(X, L\otimes A^{\vee})\rightarrow H^0(X,L)$.
Observe also that $H^0(X,A)\subseteq H^0(X,E)$ is clearly isotropic,
so it remains to determine when it is a  separable subspace.

Assume  the map $\mu$ is not injective. If $H^0(X, A)=\langle t_1, t_2\rangle$, via the
Base Point Free Pencil Trick \cite[p.~126]{ACGH}, every element in $\ker(\mu)$ is of the
form $t_1\otimes \bigl(t_2\cdot (-u)\bigr)+t_2\otimes \bigl(t_1\cdot u\bigr)$,
where $0\neq u\in H^0(X,L-2A+F)$. If
\[
\begin{tikzcd}[column sep=18pt]
m\colon H^0(X, L-A)\otimes H^0(X, \omega_X-A)\ar[r]& H^0(X, \omega_X+L-2A)
\end{tikzcd}
\]
denotes the multiplication map, then
\[
\im\Bigl\{j\colon H^0(E)\rightarrow H^0(L\otimes A^{\vee})\Bigr\}=
\Big\{s\in H^0(L\otimes A^{\vee}):
m\bigl(s\otimes H^0(\omega_X\otimes A^{\vee})\bigr)\subseteq \ker(e)\Bigr\},
\]
where the extension class $e$ is viewed as an element of $H^0(X,\omega_X+L-2A)^{\vee}$.
It follows that if $H^0(X,A)$ is not separable, then
$\im(\mu_A)\cdot \langle u\rangle \subseteq \ker(e)$, and this completes the proof.
\end{proof}

\begin{corollary}
\label{cor:pencilsrig}
Let $E$ be a rank $2$ vector bundle on $X$ as before, such that $\RR(X,E)$ is
strongly isotropic. If $A\hookrightarrow E$ is a subpencil with $h^0(X,A)=2$,
then $h^0(X,E\otimes A^{\vee})=1$.
\end{corollary}

\begin{proof}
We consider an extension such as the one in Lemma \ref{lemma:strongis}, with $h^0(X,A)=2$;
we may clearly assume that $A$ is base point free. Assuming $h^0(X,E\otimes A^{\vee})\geq 2$,
we obtain the non-injectivity of the coboundary map $\partial\colon H^0(X,L-2A)\rightarrow H^0(X, \omega_X)^{\vee}$ obtained by twisting the extension (\ref{eq:extrk2}) by $A^{\vee}$ and taking
cohomology in the corresponding exact sequence. It follows that there exists a non-zero element
$u\in H^0(X, L-2A)$ such that $u \cdot H^0(X, \omega_X)\subseteq \ker(e)$, where the left hand
side is viewed as a subspace of $H^0(X, \omega_X+L-2A)$ via the multiplication map. Applying
Lemma \ref{lemma:strongis}, this contradicts the fact that $H^0(X,A)$ is a strongly isotropic component
of $\RR(X,E)$.
\end{proof}

\begin{lemma}
\label{lemma:nohigher}
Let $X$ be a general curve of genus $g$ and $L\in \mathrm{Pic}^d(X)$ be a line
bundle of degree $d\leq 3g+1$. Then a general vector bundle $E\in \SU_X(2,L)$
carries no line subbundles $A\hookrightarrow E$ satisfying either
$\deg(A)>g+1$ or $h^0(X,A)>2$.
\end{lemma}

\begin{proof}
Suppose we are given an extension
$0\rightarrow A\rightarrow E\rightarrow L\otimes A^{\vee}\rightarrow 0$,
where $A$ is a line bundle of degree $a$. Using results of Laumon \cite{La}
(see also \cite{LN}), the generic bundle $E\in \SU_X(2,L)$ is very stable in which case
the maximal degree of a line subbundle of $E$ equals $\lfloor \frac{d+g-1}{2}\rfloor -g+1$.
Since $d\leq 3g+1$, we quickly obtain $a\leq g+1$.

In order to deal with the second statement we perform a parameter count. Assume that for
every $E\in \SU_X(2,L)$, we have an extension as above with $h^0(X,A)=r+1\geq 3$,
that is, with $A\in W^r_a(X)$. Thus $E$ corresponds to an extension class
$e\in \P \Ext^1(L\otimes A^{\vee}, A)$. If $h^0(X,2A-L)\leq 1$, then, by
Riemann--Roch, $h^0(X, \omega_X+L-2A)\leq g+d-2a$. It follows that
the number of parameters on which vector bundles $E$ appearing as
such extensions is bounded above by
\begin{align*}
\dim \P \Ext^1\bigl(L\otimes A^{\vee}, A)+\dim W^r_a(X)
&\leq g-1+d-2a+g-(r+1)(g-a+r)\\
&=-r^2-1+(a-g-1)r+g-a+d\\
&\leq 4g-a-r^2-(g+1-a)r\\
&\leq 2g+a-6 \leq 3g-5 < \dim \SU_X(2,L).
\end{align*}

If, on the other hand, $h^0(X, 2A-L)\geq 2$, we can apply Clifford's inequality and
write $h^0(X, \omega_X+L-2A)\leq g-a+\frac{d}{2}$; since $d\leq 3g+1$, one
again obtains that a general stable vector bundle $E\in \SU_X(2,L)$
does not appear in this way.
\end{proof}

\subsection{The variety of subpencils of a vector bundle}
\label{subsec:var-sub}
Before stating the next result, recall that
$\theta\in H^2\bigl(\Pic^a(X),\Q\bigr)$ is the class of the
theta divisor, see \cite[Chapter 7]{ACGH}.

\begin{proposition}
\label{prop:twistedBN}
Let $X$ be a general curve of genus $g$. Fix a positive integer $d$, a line bundle
$L\in \Pic^d(X)$, and a vector bundle $E\in \SU_X(2,L)$.

\begin{enumerate}[itemsep=3.5pt]
\item \label{bi}
For $\frac{d-2g+2}{2}\leq a\leq g+1$, each irreducible component of the variety of
subpencils
\[
W^1_a(E)\coloneqq \Bigl\{A\in W^1_a(X): h^0(X, E\otimes A^{\vee})\geq 2 \Bigr\}
\]
has dimension at least $d-3g-1$.

\item  \label{bii}
If $W^1_a(E)$ is of the expected dimension $d-3g-1$, its cohomology class equals
\[
[W^1_a(E)]=2^{2g+2a-d-1}\frac{\theta^{4g-d+1}}{(2g+2a-d-1)!\cdot (g-a+1)!\cdot
(g-a+2)!}\in H^*\bigl(\Pic^a(X),\Q\bigr).
\]

\item  \label{biii}
If $E$ is very stable and $\RR(X,E)$ is strongly isotropic, let  $A\in W^1_a(E)$ be such
that $h^0(X,A)=2$. Then $W^1_a(E)$ is smooth of dimension $d-3g-1$ at the point $[A]$ if and only if
there exists no element $0\neq \upsilon \in H^0(X,A)\otimes H^0(X, \omega_X\otimes A^{\vee})$
such that
\begin{equation}
\label{eq:smoothness2}
\upsilon \cdot H^0(X, L-2A)\subseteq \ker(e).
\end{equation}
\end{enumerate}
\end{proposition}

Note that the left hand side in \eqref{eq:smoothness2} is viewed via the multiplication map as
a subspace of $H^0(X, \omega_X+L-2A)$. Also note that Theorem \ref{thm:sepvb2} asserts
that the hypotheses on the vector bundle $E$ appearing in part (3)
of Proposition \ref{prop:twistedBN} are satisfied for a general vector bundle
$E\in \SU_X(2,L)$.

\begin{proof}
Denoting by $\PP$ the restriction to $X\times W^1_a(X)$ of the Poincar\'{e} bundle on
$X\times \Pic^a(X)$, we have that $\PP|_{X\times \{A\}}\cong A$,
for every $A\in W^1_a(X)$. Let
\[
\pi_1\colon X\times W^1_a(X)\rightarrow X \: \mbox{ and }\:
\pi_2\colon X\times W^1_a(X)\rightarrow W^1_a(X)
\]
be the two projection maps.

We fix an effective divisor $D\coloneqq p_1+\cdots+p_{b}$  of large degree
$b\coloneqq \deg(D)\gg 0$  and consider the following vector bundles
over $W^1_a(X)$,
\[
\EE\coloneqq \bigl(\pi_2\bigr)_*\Bigl(\pi_1^*\bigl(E(D)\bigr)\otimes \PP^{\vee}\Bigr)
\: \mbox{ and } \:
\FF\coloneqq\bigl(\pi_2)_*\Bigl(\pi_1^*(E(D))\otimes \PP^{\vee}|_{D}\Bigr).
\]
Note that $\rk(\EE)=d-2a+2-2g+2b$ and $\rk(\FF)=2b$ and that our
assumption on $a$ amounts to the inequality $\rk(\EE)\leq \rk(\FF)$.
There exists a vector bundle morphism $\chi\colon \EE\rightarrow \FF$ which fibrewise
corresponds to the evaluation map
\[
\begin{tikzcd}[column sep=18pt]
H^0\Bigl(X, E(D)\otimes A^{\vee}\Bigr)\ar[r]&
H^0\Bigl(X, E(D)\otimes A^{\vee}|_{D}\Bigr).
\end{tikzcd}
\]
Then we can realize $W^1_a(E)$ as the locus where $\chi$ fails to be injective.
Using the general theory of degeneracy loci, cf.~\cite[Chapter 3]{ACGH}, we
conclude that each component of $W^1_a(E)$ has dimension at least
\begin{align*}
\dim  W^1_a(X)-\rk(\FF)+\rk(\EE)-1
&=2a-g-2-2b+(d-2a+2-2g+2b)-1\\=d-3g-1.
\end{align*}
This proves part \eqref{bi}.

Applying the Porteous formula (see \cite[Theorem 4.2]{ACGH}), we compute the
virtual class of $W^1_a(E)$ (which equals its actual cohomology class when
$\dim W^1_a(E)=d-3g-1$ as expected), and we find that
\[
\bigl[W^1_a(E)]^{\mathrm{virt}}=c_{2g+2a-d-1}(\FF-\EE)
=c_{2g+2a-d-1}(-\EE),
\]
where we used that $\FF$ is algebraically equivalent to the trivial bundle over $W^1_a(X)$,
and thus $c(\FF)=1$, see \cite[p.~309]{ACGH}. Furthermore,
$c(-\EE)=e^{\mathrm{rk}(E)\cdot \theta}=e^{2\theta}$,
whereas it is well-known that
\[
[W^1_a(X)]=\frac{1}{(g-a+1)!\cdot (g-a+2)!} \, \theta^{2g-2a+2}\in
H^{2(2g-2a+2)}\bigl(\Pic^d(X), \Q\bigr),
\]
see \cite[Theorem 4.4]{ACGH}. Putting all these facts together, we obtain the equalities
\begin{align*}
\bigl[W^1_a(E)\big]^{\mathrm{virt}}&=c_{2a+2g-d-1}\bigl(e^{2\theta}\bigr)\big|_{W^1_a(X)}\\
&=\frac{2^{2a+2g-d-1}}{(2g+2a-d-1)!}\cdot \frac{\theta^{4g-d+1}}{(g-a+1)!\cdot (g-a+2)!}.
\end{align*}
This finishes the proof of part \eqref{bii}.

We now proceed to prove \eqref{biii} and begin by describing the tangent space
of $W^1_a(E)$ at a point $[A]$ using the usual identification
$T_{[A]}\bigl(\mathrm{Pic}^d(X)\bigr)\cong H^1(X, \mathcal{O}_X)\cong
H^0(X,\omega_X)^{\vee}$. From Brill--Noether theory we know that
$T_{[A]}(W^1_a(X))=\bigl(\im(\mu_A)\bigr)^{\perp}\subseteq H^1(X, \mathcal{O}_X)$,
see \cite[Chapter 4]{ACGH}. Consider the cup product map
\[
\begin{tikzcd}[column sep=20pt]
\cup\colon H^0(X, E\otimes A^{\vee})\otimes H^1(X, \mathcal{O}_X)\ar[r]
& H^1(X, E\otimes A^{\vee}),
\end{tikzcd}
\]
write $S\coloneqq \Spec\bigl(\k[t]/(t^2)\bigr)$, and denote by $\mathcal{A}$ the line
bundle on $X\times S$ corresponding to the deformation of the line bundle $A$
parametrized by a tangent vector $\varphi\in H^1(X, \mathcal{O}_X)$. Then via a
Kodaira--Spencer argument (see also \cite[\S 3.2]{HHN}), the section
$s\in H^0(X, E\otimes A^{\vee})$ can be extended to a section
$0\neq \widetilde{s}\in H^0(X\times S, E\otimes \mathcal{A}^{\vee})$
if and only $s\cup \varphi=0$.  We conclude that the tangent space
at the point $[A]$ of $W^1_a(E)$ consists of those vectors
$\varphi\in T_{[A]}\bigl(W^1_a(X)\bigr)$ such that
$s\cup \varphi=0\in H^1(X, E\otimes A^{\vee})$, for every
$s\in H^0(X,E\otimes A^{\vee})$. Via Serre duality, we have that
$H^1(X,E\otimes A^{\vee})\cong H^0\bigl(X,E^{\vee}\otimes \omega_X\otimes A\bigr)^{\vee}$;
therefore, we obtain
\begin{equation}
\label{eq:tgspace}
T_{[A]}\bigl(W^1_a(E)\bigr)=\Bigl\{\varphi \in H^0(X,\omega_X)^{\vee}:
\varphi|_{H^0(X,A)\otimes H^0(X,\omega_X\otimes A^{\vee})+\im(\beta_{E,A})}=0\Bigr\},
\end{equation}
where $\beta_{E,A}$ is the twisted Petri map introduced in (\ref{eq:twPetri}).

We continue with the proof of \eqref{biii} and assume  $E\in \SU_X(2,L)$ is a very stable
vector bundle such that $\RR(X,E)$ is strongly isotropic and choose $[A]\in W^1_a(E)$
to be an element such that $h^0(X,A)=2$. Using Corollary \ref{cor:pencilsrig}, we have that
$h^0(X, E\otimes A^{\vee})=1$, whereas the very stability of $E$ guarantees
that the map $\beta\coloneqq \beta_{E,A}$ is in fact injective; in particular,
\[
\dim  \im(\beta)\geq h^0(X, E^{\vee}\otimes \omega_X\otimes A)=2a+2g-d-1.
\]
Since $X$ is general, the Petri map $\mu_A$ is injective, hence
$\dim \im(\mu_A)=2(g-a+1)$ and
\[
\dim \im(\beta)+\dim \im(\mu_A)=4g-d+1.
\]
Thus $W^1_a(E)$ is smooth of the expected dimension $d-3g-1$ at the point $[A]$ if and only if
\begin{equation}
\label{eq:int}
\im(\beta)\cap \im(\mu_A)=0.
\end{equation}

Assume \eqref{eq:int} does not hold, in which case there are non-zero sections
$s\in H^0(X, E\otimes A^{\vee})$ and $t\in H^0\bigl(X, E^{\vee}\otimes \omega_X\otimes A)$
such that $\beta(s\otimes t)\in H^0(X,A)\otimes H^0(X, \omega_X\otimes A^{\vee})\subseteq H^0(X, \omega_X)$.

In order to make this condition more explicit, we write down the exact sequence obtained by tensoring \eqref{eq:extrk2}, using that  $E\cong E^{\vee}(L)$ and then taking cohomology:

\begin{tikzcd}[column sep=20pt]
0\ar[r]&  H^0\bigl(X, E^{\vee}\otimes \omega_X\otimes A\bigr)  \ar[rr, "\beta(s\otimes -)"]
&& H^0(X, \omega_X)  \ar[r, "\pd"] & H^0(X, L-2A)^{\vee}
\end{tikzcd}

Assume $\beta(s\otimes t)\in H^0(X,A)\otimes H^0(X, \omega_X\otimes A^{\vee})$ for
$t\in H^0\bigl(X, E^{\vee}\otimes \omega_X\otimes A\bigr)$. We set
$\upsilon\coloneqq \beta(s\otimes t)\in H^0(X,A)\otimes H^0(X,\omega_X\otimes A^{\vee})$
and denote by
\begin{equation}
\label{eq:nu}
\begin{tikzcd}[column sep=20pt]
\nu\colon \Bigl(H^0(X,A)\otimes H^0(X, \omega_X\otimes A^{\vee})\Bigr)\otimes
H^0(X,L-2A)\ar[r]& H^0(X,\omega_X)
\end{tikzcd}
\end{equation}
the multiplication map, where we identify $H^0(X,A)\otimes H^0(X, \omega_X\otimes A^{\vee})$
with its image under the map $\mu_A$.  We then obtain that $W^1_a(E)$ is smooth of the expected
dimension at the point $[A]$ if and only if there exist no element
$0\neq \upsilon \in H^0(X,A)\otimes H^0(X,\omega_X\otimes A^{\vee})$ such that
\begin{equation}
\label{eq:smooth_twBN}
\nu\bigl(\upsilon \cdot H^0(X, L-2A)\bigr)\subseteq \ker(e),
\end{equation}
where recall that $e\in \Ext^1(L\otimes A^{\vee}, A)$ is the extension class of the vector bundle $E$.
\end{proof}

\begin{remark}
\label{rem:w1a}
The argument given in Proposition \ref{prop:twistedBN} also shows that when
$a<\frac{d-2g+2}{2}$, the equality $W^1_a(E)=W^1_a(X)$ holds.
\end{remark}

The following result shows that in the extremal case $d=3g+1$, the strong isotropicity of the
resonance $\RR(X,E)$ is equivalent to the smoothness of all determinantal loci $W^1_a(E)$.

\begin{theorem}
\label{thm:equivalences}
Let $X$ be a general curve of genus $g$ and let $E\in \SU_X(2,L)$ be a
very stable vector bundle of degree $3g+1$.

\begin{enumerate}[itemsep=2pt,topsep=-1pt]
\item \label{ist1}
If $\RR(X,E)$ is strongly isotropic, then $W^1_a(E)$ is smooth and zero-dimensional
for each $\frac{g+2}{2}\leq a\leq g+1$.

\item \label{ist2}
The variety $W^1_{g+1}(E)$ is smooth, zero-dimensional, and consists of $2^g$
reduced points.
\end{enumerate}
\end{theorem}

\begin{proof}
Assume $\RR(X,E)$ is strongly isotropic. Since $E$ is very stable $H^1(X, L-2A)=0$,
for every subpencil $A\in W^1_a(E)$, therefore by Riemann--Roch we obtain
$h^0(X, L-2A)=2(g-a+1)$. Moreover, the argument in Lemma \ref{lemma:nohigher}
shows that $h^0(X,A)=2$, for every such $A$.  Set
$W\coloneqq H^0(X,A)\otimes H^0(X, \omega_X\otimes A^{\vee})$ and
$U\coloneqq H^0(X,L-2A)$, therefore $\dim(U)=\dim(W)=2(g-a+1)$.
We consider the following diagram,
\begin{equation}
\begin{tikzcd}[column sep=20pt,row sep=21pt]
\P(W)\times \P(U)\cong \P^{2g-2a+1}\times \P^{2g-2a+1}  \ar[d, "\pi_1"]
 \ar[dr, "\pi_2"]  \ar[r, "\nu"] & \P\bigl(H^0(X, \omega_X+L-2A)\bigr)  .\\
\P(W) & \P(U)
\end{tikzcd}
\end{equation}
Here $\nu$ is the projection of the Segre embedding induced by the multiplication
map, and thus
\[
\nu^*\bigl(\mathcal{O}_{\P(H^0(\omega_X+L-2A))}(1)\bigr)=\pi_1^*\bigl(\mathcal{O}_{\P(W)}(1)\bigr)\otimes \pi_2^*\bigl(\mathcal{O}_{\P(U)}(1)\bigr).
\]
Assume by contradiction that $W^1_a(E)$ is not smooth and of dimension zero at a point $[A]$.
Then by applying Proposition \ref{prop:twistedBN}, there exists an element $[v]\in \P(W)$ such that
$v\cdot H^0(X,L-2A)\subseteq \ker(e)$, where recall that $\ker(e)$ is regarded as a hyperplane inside
$H^0(X,\omega_X+L-2A)$. In other words, if $H$ is the pull-back of $\ker(e)$ under the natural map
$\P(W\otimes U)\dashrightarrow \P\bigl(H^0(X, \omega_X+L-2A)\bigr)$, writing the equation of $H$
as
\[
\sum_{1\le i,j \le 2(g-a+1)} c_{ij} z_{ij}=0,
\]
it follows that the coefficient matrix $(c_{ij})$ is singular.
Therefore, there exists an element $[u]\in \P(U)$ such that
$\nu\bigl(\P(W)\times \{[u]\}\bigr)\subseteq \P \ker(e)$.
Using Lemma \ref{lemma:strongis}, this amounts to $\RR(X,E)$ not being strongly isotropic.
This contradicts our hypothesis, and thus proves part \eqref{ist1}.

Part \eqref{ist2}  follows immediately from part \eqref{ist1}, once we consider the tangent space
description \eqref{eq:tgspace} when $d=3g+1$. Since $E$ is very stable, the map $\beta$ is
surjective, hence condition \eqref{eq:int} is automatically satisfied. This completes the proof.
\end{proof}

\begin{remark}
\label{rem:no-line}
Assuming $E\in \SU_X(2,L)$ is a vector bundle as above
(of degree $d=3g+1$) which contains no line subbundles
$A\hookrightarrow E$ with $h^0(X,A)> 2$ (an assumption satisfied outside a subset
of codimension $3$ in $\SU_X(2, L)$, see Lemma \ref{lemma:nohigher}), the proof
of Theorem \ref{thm:equivalences} can be reversed and we obtain that $\RR(X,E)$ is strongly
isotropic if and only if $W^1_a(E)$ is smooth and zero-dimensional for each $\frac{g+2}{2}\leq a\leq g+1$.
\end{remark}

\subsection{The resonance variety of a general vector bundle}
\label{subsec:res-gen-bundle}
We are now in a position to complete the proof of Theorem \ref{thm:sepvb2}
from the Introduction.

\noindent \emph{Proof of Theorem \ref{thm:sepvb2}.}
We fix a general vector bundle $E\in \SU_X(2,L)$ of degree  $d\leq 3g+1$.
Assume by contradiction that $\RR(X,E)$ is not strongly separable, which in this case amounts
to its non-separability. After invoking Lemma \ref{lemma:nohigher}, we can express $E$
as an extension
\begin{equation*}
\begin{tikzcd}[column sep=20pt]
0\ar[r]& A \ar[r]&  E \ar[r]& L\otimes A^{\vee} \ar[r]&  0
\end{tikzcd}
\end{equation*}
as in Lemma \ref{lemma:strongis}, where $A\in W^1_a(X)$ is such that $h^0(X,A)=2$
and the base locus $F$ of $\abs{A}$
is an effective divisor of degree $b$ on $X$ and the subspace
$H^0(X,A)\subseteq H^0(X,E)\equalscolon V^{\vee}$ is a non-separable
component of $\RR(X,E)$. Therefore, the conclusion of Lemma \ref{lemma:strongis}
holds. Via a parameter count, we shall argue that this is not possible for a general
choice of $E$.

For given values $\frac{g+2}{2}\leq a\leq g+1$ and $0\leq b\leq a$, we let $X_b$
be the $b$-th symmetric product of $X$, and we denote by $\mathcal{T}_{a,b}$
the subvariety of $W^1_a(X)\times X_b$ consisting of pairs $(A,F)$, where $F$
is an effective divisor of degree $b$ on $X$ and $A\in W^1_a(X)$ is a pencil
having $F$ in its base locus. Clearly, if $\mathcal{T}_{a,b}$ is non-empty, it is
irreducible of dimension $\dim W^1_{a-b}(X)+b=2a-g-2-b$, where we use the
generality of $X$, which implies that $\dim W^1_{a-b}(X)=\max \{2a-2b-g-2,0\}$.

We now introduce the parameter space
\[
\Sigma_{a,b}\coloneqq
\left\{ \bigl( A, F, [u], [e] \bigr):
\renewcommand*{\arraystretch}{1.2}
\begin{array}{l}
A\in W^1_{a,F}(X), \ F \in X_b \\
\left[u\right] \in \P H^0(X, L-2A+F), \
\left[e\right] \in \P H^0(X, \omega_X+L-2A)^{\vee}\\
H^0(X, A)\otimes H^0(X, \omega_X\otimes A^{\vee})\cdot \langle u\rangle \subseteq \ker{(e)}
\end{array}
\renewcommand*{\arraystretch}{1}
\hspace*{-4pt}\right\},
\]
together with the projections
\begin{equation}
\label{eq:incidence}
\begin{tikzcd}[column sep=28pt]
 \mathcal{T}_{a,b}
& \Sigma_{a,b} \ar[l, "\:\pr_1"'] \ar[r, "\pr_2"]
&\SU_X(2,L),
\end{tikzcd}
\end{equation}
where $\mathrm{pr}_2$ associates to $\bigl(A, F, [u], [e]\bigr)$ the vector bundle $E$
corresponding to the extension class $e\in \Ext^1(L\otimes A^{\vee}, A)$, whereas
$\mathrm{pr}_1\bigl(A,F,[u],[e]\bigr)\coloneqq (A,F)$. We now estimate the general fibre dimension
of $\mathrm{pr}_1$. By the generality hypothesis, $E$ may be assumed to be very stable,
therefore $h^1(X, \omega_X+2A-L)=0$, hence $h^0(X, L-2A+F)=d-2a+b+1-g$.
Having fixed $A, F$, and $[u]\in \P H^0(X,L-2A+F)$, the parameter space for
extension classes $[e]\in \P \mathrm{Ext}^1(L\otimes A^{\vee}, A)$ such that
$\bigl(A,F, [u], [e]\bigr)\in \Sigma_{a,b}$ is the projective space of dimension
\[
h^0(X, \omega_X+L-2A)-h^0(X, A)\cdot h^0(X, \omega_X\otimes A^{\vee})-1=d-g-4.
\]
We obtain the estimate
\begin{align*}
\dim \Sigma_{a,b}\leq \dim \mathcal{T}_{a,b}+(d-2a+b-g)+(d-g-4)\\
=2d-3g-6\leq 3g-4< \dim \SU_X(2,L),
\end{align*}
which shows that the resonance of a general vector bundle $E\in \SU_X(2,L)$
is strongly isotropic.

In the case $d\leq 3g$, this parameter count also establishes that $\RR(X,E)=0$,
for a general $E\in \SU_X(2,L)$. Indeed, assuming the general vector
bundle $E$ appears in an extension (\ref{eq:extrk2}), we obtain that vector bundles
appearing in this way depend on at most
\[
\dim W^1_a(X)+h^0(X,\omega_X+L-2A)-1=d-4\leq 3g-4<3g-3=\dim \SU_X(2,L)
\]
parameters, thus finishing the proof.
\hfill $\Box$

\begin{remark}
\label{rmk:nonruled}
In degree
$3g+2\leq d\leq  4g$, the projective resonance $\mathbf{R}_a(X,E)$ is not linear. Indeed,
it follows from Propositions \ref{prop:max} and \ref{prop:twistedBN} that $\mathbf{R}_a(X,E)$
admits a regular fibration over $W^1_a(E)$. On the other hand, since $\theta$ is an ample
class on the Jacobian variety $\Pic^a(X)$, it follows from part \eqref{bii} of
Proposition \ref{prop:twistedBN} that each component of $W^1_a(E)$ is a positive
dimensional variety of general type; thus, $\mathbf{R}_a(X,E)$ cannot be linear.		
\end{remark}

\section{K\"{a}hler groups and Kodaira fibrations}
\label{subsect:Kahler}

\subsection{Resonance varieties of K\"{a}hler manifolds}
\label{subsec:res-kahler}
We now discuss the case of K\"{a}hler groups, when the resonance
varieties in question are {\em not}\/ isotropic (unless they vanish).
For a compact K\"{a}hler manifold $X$, let
\[
\begin{tikzcd}[column sep=19pt]
\cup_X\colon \bwedge^2 H^1(X, \C)\ar[r]& H^2(X,\C)
\end{tikzcd}
\]
be the cup product map. We consider the resonance variety
$\RR(X)\coloneqq \RR\bigl(\pi_1(X)\bigr)$
of the fundamental group of $X$. As is well-known, $X$ is formal, and thus $\pi_1(X)$
is $1$-formal; thus, $\RR(X)$ is linear and projectively disjoint. On the other hand,
if $\RR(X)\ne \{0\}$, then, as shown in \cite[Corollary 7.3]{DPS-duke}, all irreducible
components of $\RR(X)$ are $1$-isotropic, that is, the restriction of $\cup_X$ to
each such component has $1$-dimensional image.

The first notable case is that of surface groups.
Let $\Sigma_g$ be a smooth algebraic curve of genus $g\ge 2$, and let
$\Pi_g\coloneqq \pi_1(\Sigma_g)$ be its fundamental group.
It is well-known that $\Sigma_g$ is formal, and thus $\Pi_g$ is
$1$-formal.  The cohomology ring has the form $H^{\hdot}(\Sigma_g)=E/I$,
where $E=\bw\bigl(e_1,\dots,e_{g},\bar{e}_1,\dots,\bar{e}_{g}\bigr)$ and $I$ is the ideal
generated by the quadrics $e_i\wedge e_j,\bar{e}_i\wedge \bar{e}_j$  for $1\le i<j\le g$,
$e_i\wedge \bar{e}_j$ for $i\ne j$, respectively
$e_1\wedge \bar{e}_1-e_i\wedge \bar{e}_i$  for $1< i \le g$.

It is readily seen that $\mc{R}(\Pi_g)=H^1(\Pi_g,\C)=\C^{2g}$.
Clearly, this linear space is $1$-isotropic and separable.
Moreover, it follows from \cite[Theorem 7.3]{PS-imrn} that
the infinitesimal Alexander invariant $\mf{B}(\Pi_g)$ has
Hilbert series
\[
\Hilb(\mf{B}(\Pi_g),t)= 1-\frac{1-2g t+t^2}{(1-t)^{2g}}.
\]
Therefore, the Chen ranks   are given by
$\theta_1(\Pi_g)=2g$, $\theta_2(\Pi_g)=2g^2-g-1$, and
\begin{equation}
\label{eq:surface-chenranks}
\theta_q(\Pi_g)=(q-1)\binom{2g+q-2}{q}-\binom{2g+q-3}{q-2} \quad\text{for $q \ge 3$}.
\end{equation}

\medskip

We discuss next the case of irregular fibrations.
A fibration from a compact complex manifold $X$ onto a smooth complex
curve $B$ is a surjective holomorphic map with connected fibers.
For smooth projective varieties, and, more generally, for compact K\"ahler
manifolds, all components of the resonance variety $\mc{R}(X)$ are of the form
$V_f\coloneqq f^*H^1(B,\C)$, where $f\colon X\to B$ runs through the (finite)
set $\mathcal{E}(X)$ of equivalence classes of fibrations with base $B$ being
a smooth curve of genus $g\ge 2$ and with no multiple fibers
(see e.g.~\cites{DPS-duke, Delzant, Su-imrn}).  We sometimes
write these maps as $f\colon X\rightarrow B_f$ to emphasize the
dependence of the base of the fibration on $f$. The resulting
components $V_f$ are linear (of dimension $2g(B)\ge 4$),
projectively disjoint, and $1$-isotropic. As shown by
Catanese \cite{Cat}, the existence of such a fibration $f\colon X\to B$
is equivalent to the existence of an epimorphism
$\pi_1(X)\surj \pi_1(B)$ whose kernel is finitely generated.

\begin{example}
\label{ex:CCML}
To illustrate the phenomenon that the resonance \emph{does not} detect fibrations with multiple fibres, we mention the  Catanese--Ciliberto--Mendes Lopes surface $X$ \cite[Example 2]{HP},  which admits an elliptic
fibration, $f\colon X\to B$ with base a smooth curve of genus $2$.
It follows from \cites{HP, Su-imrn} that the resonance variety $\mc{R}(X)$
consists of a single, $4$-dimensional linear subspace in $H^1(X,\C)=\C^6$,
which is equal to $f^*(H^1(B,\C))$. Indeed, direct computation shows that $H^*(X,\C)=E/I$,
where $E=\bigwedge(e_1,\bar{e}_1,e_2,\bar{e}_2,e_3,\bar{e}_3)$ and
$I=\langle e_1\wedge  e_2, \bar{e}_1\wedge \bar{e}_2, e_1\wedge \bar{e}_2,
e_2\wedge \bar{e}_1, e_1\wedge \bar{e}_1-e_2\wedge \bar{e}_2 \rangle$.
Setting $V^{\vee}=H^1(X,\C)$ and $K^{\perp}=I_2\subseteq \bw^2 V^{\vee}$,
we have that $\mc{R}(X)=\mc{R}(V,K)=\overline{V}^{\vee}=
\langle e_1,\bar{e}_1,e_2,\bar{e}_2 \rangle$.
Clearly, $K^{\perp}\subset \bw^2\overline{V}^{\vee}$ and
$\ol{K}=\langle v_1\wedge \bar{v}_1\rangle$;
thus, $\overline{V}^{\vee}$ is $1$-isotropic. Moreover,
$\overline{V}^{\vee}$ is separable, and thus, by
Theorem \ref{thm:sep->disjoint} it defines an isolated component
of $\Rproj(V,K)$. To sum up, the resonance of $X$ is linear, projectively disjoint,
and reduced. The surface $X$ also admits a
fibration with base an elliptic curve and with $2$ singular fibers, each of
multiplicity $2$. This (singular) fibration defines by pullback a $2$-dimensional
translated torus component of the characteristic variety $\mc{V}(X)$, but this
component is not detected by $\mc{R}(X)$, see again \cite{Su-imrn}.
\end{example}

\subsection{Kodaira fibrations}
\label{subsect:kodfib}
Assume now that $f\colon X\rightarrow B$ is a Kodaira fibration over a smooth
projective curve of genus $b\geq 2$ and denote by $\Sigma_g$ a general fibre of $f$.
The study of such fibrations goes back to the fundamental papers of Atiyah and
Kodaira \cite{At, Kod}. A Kodaira fibration induces a \emph{monodromy representation},
$\rho\colon \Pi_b \rightarrow \Sp_{2g}(\Z)=\Aut \bigl(H_1(\Sigma_g, \Z)\bigr)$.
The Leray--Serre spectral sequence of the fibration gives rise to the short exact sequence
\begin{equation}
\label{eq:invmod}
\begin{tikzcd}[column sep=20pt]
0 \ar[r]& H^1(B, \C)\ar[r, "f^*"]& H^1(X,\C) \ar[r]& H^1(\Sigma_g,\C)^{\rho}\ar[r]&  0.
\end{tikzcd}
\end{equation}

In particular, the \emph{relative irregularity}, $q_f\coloneqq q(X)-g(B)$ of the fibration
can be regarded as the invariant part of the monodromy representation; that is,
$ q_f=\frac{1}{2} h^1\bigl(\Sigma_g, \C\bigr)^{\rho}$.  There are essentially two
known ways to construct Kodaira fibrations. One is by taking generic complete
intersections in the moduli space $\overline{\mathcal{M}}_g$ of stable curves
of genus $g$ using that the Satake compactification of $\mathcal{M}_g$ has
boundary of codimension $2$; the other via ramified branched cover
constructions over product of curves, see \cites{At, CR, Kod}. No examples
of compact K\"ahler surfaces $X$ having at least three non-equivalent Kodaira
fibrations are known, see \cite[Question 10]{C2}. On the other hand, Salter \cite{Salter}
has provided examples of closed (non-algebraic) $4$-manifolds which admit a
number of non-equivalent surface bundle structures that is an arbitrary power of $2$.

The resonance of complete intersection Kodaira fibrations turns out to be quite simple.

\begin{lemma}
\label{lemma:complint}
Let $\Sigma_g\hookrightarrow X\stackrel{f}\twoheadrightarrow B$ be a complete
intersection Kodaira fibration. Then $H^1(X,\C)\cong f^*H^1(B,\C)$ and accordingly
the resonance $\RR(X)=f^*H^1(B,\C)$ is separable.
\end{lemma}

\begin{proof}
We let $\mathcal{M}_g^{(n)}$ be the moduli space of genus $g$ curves
with a level $n\geq 3$ structure, that is, the parameter space for smooth
curves of genus $g$, together with the choice of a symplectic isomorphism
$\Pic^0(X)[n]\cong \bigl(\Z/n\Z)^{2g}$. It is known that $\mathcal{M}_g^{(n)}$
is a fine moduli space of curves; we denote by $\overline{\mathcal{M}}_g^{(n)}$ the
normalization of the Deligne--Mumford compactification $\overline{\mathcal{M}}_g$
in the function field of $\mathcal{M}_g^{(n)}$ (see \cite{vGO} for details on these matters).

Using the theory of Satake compactifications, there exists a regular map,
$\iota\colon \overline{\mathcal{M}}_g^{(n)}\rightarrow \mathbf{P}^N$
such that the boundary $\overline{\mathcal{M}}_g^{(n)}\setminus \mathcal{M}_g^{(n)}$
is contracted to a codimension $2$ set of its image. It follows that the inverse image
under $\iota$ of the intersection of $3g-4$ general hyperplanes in $\mathbf{P}^N$ is
a smooth projective curve $B\subseteq \mathcal{M}_g^{(n)}$. Since $\mathcal{M}_g^{(n)}$
is a fine moduli space, the pull-back to $B$ of its universal family induces a Kodaira
fibration $f\colon X\rightarrow B$ of smooth curves of genus $g$. By the Lefschetz
hyperplane section theorem, since $B$ is obtained by intersecting
$\overline{\mathcal{M}}_g^{(n)}$ with ample divisors, the map
$\iota_*\colon \pi_1(B)\rightarrow \pi_1\bigl(\mathcal{M}_g^{(n)}\bigr)$
is surjective. Consequently, the image under the monodromy map of
$\pi_1(B)$ inside the mapping class group $\mathrm{Mod}_g$ is
of finite index. It follows that $H^1(\Sigma_g,\C)^{\rho}=0$, and so
by the exactness of \eqref{eq:invmod}, we have that $H^1(X,\C)\cong f^*H^1(B, \C)$.
Therefore, $\ker(\cup_X)\cong \ker(\cup_B)$ and accordingly $\RR(X)\cong f^*\RR(B)$.
Applying Lemma \ref{lem:alt-sep}, we conclude that $\RR(X)$ is separable.
\end{proof}

It follows from the lemma that formula \eqref{eq:surface-chenranks} applies for
the Chen ranks of $\pi_1(X)$.

Assume now $X$ admits two non-equivalent Kodaira fibration structures,
$\Sigma_{g_1}\hookrightarrow X\stackrel{f_1}\rightarrow B_1$ and
$\Sigma_{g_2}\hookrightarrow X\stackrel{f_2}\rightarrow B_2$.
We consider the product map, $f=(f_1, f_2)\colon X\longrightarrow B_1\times B_2$,
and the homomorphisms induced in cohomology,
\begin{equation}
\label{eq:pullbackhom}
\begin{tikzcd}[column sep=19pt]
H^i(f)\colon H^i(B_1\times B_2, \C)\ar[r]& H^i(X, \C)
\end{tikzcd}
\end{equation}
for $i=1,2$.

\noindent \emph{Proof of Theorem \ref{thm:koddouble}}.
With the notation of \eqref{eq:pullbackhom}, we are in the
situation of a double Kodaira fibration for which $H^1(f)$ is
an isomorphism, whereas $H^2(f)$ is injective. We are going
to verify the separability of $\RR(X)$ using Lemma \ref{lem:alt-sep}.
To that end, we write down the following commutative diagram:
\begin{equation}
\label{eq:com_product}
\begin{tikzcd}[column sep=39pt]
 \bigwedge^2 H^1\bigl(B_1\times B_2, \C\bigr) \ar[r, "\bwedge^2 H^1(f)"]
 \ar[d, "\cup_{B_1\times B_2}"]
&\bigwedge^2 H^1(X, \C) \ar[d, "\cup_X"]\\
H^2\bigl(B_1\times B_2, \C\bigr) \ar[r, hook, "H^2(f)"] & H^2(X,\C)
\end{tikzcd}
\end{equation}
From the K\"unneth formula, it follows that
$\ker(\cup_X)\cong \ker(\cup_{B_1})\oplus \ker(\cup_{B_2}\bigr)$,
and accordingly,  $\RR(X)=f_1^*H^1(B_1,\C)\cup f_2^*H^1(B_2,\C)$.
Setting $U_i \coloneqq \ker\bigl\{(f_i)_*\colon
H_1(X,\C)\rightarrow H_1(B_i,\C)\bigr\}$, we obtain that
\begin{equation}
\label{eq:K2}
K=\bigl(U_1\otimes U_2\bigr)\oplus \bigl(\bwedge^2 f_*\bigr)^{-1}\bigl(H_2(B_1, \C)\bigr)\oplus
\bigl(\bwedge^2 f_*\bigr)^{-1}\bigl(H_2(B_2,\C)\bigr)\subseteq \bwedge ^2 H_1(X, \C).
\end{equation}
Using the notation (\ref{eq:decomp-lmod}) for the component
$\overline{V}_1^{\vee}=f_1^*\bigl(H^1(B_1,\C)\bigr)$ of $\RR(X)$, we have that
\[
H=\bwedge^2 U_1,  \  M=U_1\otimes U_2 ,\ \mbox{and } \ L=\bwedge^2 U_2.
\]
Since $U_1\otimes U_2\subseteq K$, it follows with the notation
of \eqref{eq:projection-m} that the map
$p_M\colon K\cap\bigl(H\oplus M\bigr)\rightarrow M$ is surjective.
Using Lemma \ref{lem:alt-sep}, we conclude that the component
$f_1^*H^1(B_1,\mathbb C)$ of $\RR(X)$ is separable. Same
considerations apply for the component $f_2^*H^1(B_2,\C)$ of the resonance.
\hfill $\Box$

\subsection{The Atiyah--Kodaira construction}
\label{subsec:a-k}
The simplest examples of double Kodaira fibrations where the
conditions of Theorem \ref{thm:koddouble} are satisfied are provided
by the Atiyah--Kodaira fibrations constructed in the 1960s in \cites{At, Kod}.
These surfaces can  be described as follows. Let $\tau\colon B_{2}\to B_{2}$
be a fixed point free involution of a curve of genus $g(B_2)=2g-1$, and let
$\varphi\colon B_1\to B_2$ be the congruence unramified cover classified by the
homomorphism $\pi_1(B_2)\surj H_1(B_2,\Z)\surj H_1\bigl(B_2,\Z/2\Z\bigr)$,
thus $\deg(\varphi)=2^{2(2g-1)}$ and $g(B_1)=1+2^{4g-1}(g-1)$.
We let $X$ be the $2$-fold branched cover of $B_1\times B_2$
ramified along the divisor $Y_1+Y_2$, where
\[
Y_1\coloneqq \bigl\{(z,\varphi(z)):z\in B_1\bigr\} \ \mbox{ and } \ \
Y_2\coloneqq \bigl\{\bigl(z, \tau(\varphi(z))\bigr):z\in B_1\bigr\}.
\]
Observe that $Y_1\cap Y_2=\emptyset$. The two independent
Kodaira fibrations of $X$ are of the form
\begin{equation}
\label{eq:fibrations}
\begin{tikzcd}[column sep=17pt]
\Sigma_{4g-2}\ar[r]& X\ar[r, "f_1"] &B_1
\end{tikzcd}
 \quad\text{and}\quad
\begin{tikzcd}[column sep=17pt]
\Sigma_{1+2^{4g-2}(4g-3)}\ar[r]& X\ar[r, "f_2"] &B_2 .
\end{tikzcd}
\end{equation}
Using \cite[Theorem 2.2]{Kas}, we obtain that
$f=(f_1,f_2)\colon X \longrightarrow B_1\times B_2$ induces
an isomorphism on $H^1(-,\C)$ and an injection on $H^2(-, \C)$, see
also \cite[Theorem 1.1]{ChenLei}, where only the case $g=2$ of this
construction is treated.
Indeed, one applies the sequence
(\ref{eq:invmod}) to the fibration $f_1\colon X\rightarrow B_1$
and use that its general fibre $\Sigma_{4g-2}$ is a double cover
of $B_2=\Sigma_{2g-1}$, branched at two points conjugate under
the involution $\tau$. Then
$H^1\bigl(\Sigma_{4g-2}, \C)^{\pi_1(B_1)}\cong H^1(B_2,\C)$,
therefore $b_1(X)=4g+2^{4g}(g-1)$ and
$\mc{R}(X)=\overline{V}_1^{\vee} \cup \overline{V}_2^{\vee}$,
where $\overline{V}_i^{\vee}=f_i^*H^1(B_i,\C)$ for $i=1,2$.
Therefore, Theorem \ref{thm:koddouble} applies in this case.

Similar considerations apply to the example of the double Kodaira fibration
given in \cite{CFT}. This complex surface $X$ has $b_1(X)=38$ and fibers in
two distinct ways, $\Sigma_4 \to X\to \Sigma_{17}$ and $\Sigma_{49} \to X\to \Sigma_{2}$.
A similar argument shows that $\mc{R}(X)$ is separable. We have not tried to
verify the hypothesis of Theorem \ref{thm:koddouble} for other families of
double Kodaira fibrations, for instance, those constructed in \cite{CR}.

These considerations naturally raise the following question.

\begin{question}
\label{quest:kahler-res}
Is the resonance variety of a K\"{a}hler group $G$ always linear, projectively disjoint,
and reduced?
\end{question}

Of course, the real question is whether $\Rproj(G)$ is reduced.
Again, it would be enough to show
that the components $f^*(H^1(B_f,\C))$ of $\mc{R}(G)$ are all separable.
If this question were to have a positive answer, the following \emph{Chen ranks
conjecture for K\"ahler groups} would be the natural next step to consider.

\begin{conjecture}
\label{conj:chen-kahler}
Let $X$ be a compact K\"ahler manifold. For each $g\geq 2$, we denote by
$h_{g}(X)\coloneqq \#\big\{f\in \mathcal{E}(X) : g(B_f)=g\big\}$
the number of components of $\mc{R}(X)$ of dimension $2g$. Then for all
$q\gg 0$, the following Chen rank formula holds:
 \begin{align*}
 \theta_q\bigl(\pi_1(X)\bigr) &= \sum_{f\in \mathcal{E}(X) }  \theta_q(\pi_1(B_f))
= \sum_{g\ge 2} h_{g}(X) \theta_q(\Pi_g) \\
 &=(q-1)\cdot \sum_{g\ge 2} h_g(X)\binom{2g+q-2}{q}
 -\sum_{g\ge 2} h_g(X)\binom{2g+q-3}{q-2}.
 \end{align*}
\end{conjecture}

\section{Resonance of right-angled Artin groups}
\label{sect:right-angled}

We now apply the general theory to the case when the subspace
$K \subseteqq  \bwedge^2 V $ admits a \emph{monomial basis},
that is, there exists a basis $\{v_1,\dots , v_n\}$
for $V$ so that $K$ admits a basis whose elements are of the form
$v_i \wedge v_j$.

Let $\binom{[n]}{k}$ be the set of ordered $k$-tuples from $[n]=\{1,\dots,n\}$.
The above information is conveniently encoded in a
(simple) graph $\Gamma=(\mathsf{V},\mathsf{E})$
on vertex set $\mathsf{V}=[n]$ and edge set
\begin{equation}
\label{eq:edges}
\mathsf{E} = \Bigl\{(i,j)\in \tbinom{[n]}{2} : v_i \wedge v_j \in K\Bigr\}.
\end{equation}
Dually, $V^{\vee}$ is spanned by $\{e_1,\dots , e_n\}$
and $K^{\perp}$ is the linear subspace spanned by the elements
$\{e_i\wedge e_j: (i,j) \in \overline{\mathsf{E}}\}$, where
$\overline{\mathsf{E}}=\binom{[n]}{2}  \setminus \mathsf{E}$ is the
edge set of the complement graph
$\overline{\Gamma}=(\mathsf{V},\overline{\mathsf{E}})$.
We also denote by $\mathsf{T}$ the set of complete triangles
in $\Gamma$, and by $\overline{\mathsf{T}}=
\binom{[n]}{3}  \setminus \mathsf{T}$ the set of triangles with
at least one missing edge. To the graph $\Gamma$ one can
associate the \emph{right-angled Artin group}\/
$G_{\Gamma}$ having the following presentation
\[
G_{\Gamma}\coloneqq \Bigl \langle v_i: v_i\cdot v_j=v_j\cdot v_i  \ \
\mbox{ for } \ v_i\wedge v_j\in K\Bigr\rangle.
\]
The cohomology ring of $G_{\Gamma}$ can be identified with the Stanley--Reisner
algebra $E/\langle K^{\vee}\rangle_E$, where $E=\bigwedge V^{\vee}$ is the
exterior algebra generated by $e_1, \ldots, e_n$.

A graph $\Gamma'=(\mathsf{V}',\mathsf{E}')$ of $\Gamma$ is a
\defi{full subgraph of $\Gamma$} (or, an induced subgraph
on the vertex set $\mathsf{V}'\subseteq \mathsf{V}$)
if $\mathsf{E}'=\mathsf{E}\cap  \binom{\mathsf{V}'}{2}$.
Given two graphs, $\Gamma'=(\mathsf{V}',\mathsf{E}')$ and
$\Gamma''=(\mathsf{V}'',\mathsf{E}'')$, their \defi{join} $\Gamma=\Gamma'*\Gamma''$,
is the graph with vertex set $\mathsf{V}=\mathsf{V}'\cup\mathsf{V}''$ and having the edge set
$\mathsf{E}=\mathsf{E'}\cup\mathsf{E}''\cup \{(i',i'') : i'\in \mathsf{V}', i''\in \mathsf{V}''\}$.

Let $S=S_{\Gamma}$ be the polynomial ring $\k[x_1,\dots,x_n]$ and
let $W_\Gamma=W(G_{\Gamma})\coloneqq W(V,K)$ be the Koszul module associated
to $\Gamma$. This module has the following graph-theoretic description:

\begin{lemma}[\cite{PS-jlms07}]
\label{lem:kg-pres}
The Koszul module of a graph $\Gamma$ admits a
presentation,
\begin{equation}
\label{eq:w-gamma}
\begin{tikzcd}[column sep=19pt]
W_\Gamma=\coker \big\{ \Theta \colon
\spn(\overline{\mathsf{T}}) \otimes S \ar[r]& \spn(\overline{\mathsf{E}}) \otimes S \bigr\} ,
\end{tikzcd}
\end{equation}
where $\Theta$ is the matrix with entries
\begin{equation}
\label{eq:Theta}
\Theta_{ijk,\ell m} =
\begin{cases}
x_k & \text{if $\ell=i$, $m=j$}\\
-x_j & \text{if $\ell=i$, $m=k$}\\
x_i & \text{if $\ell=j$, $m=k$}\\
0 & \text{otherwise}.
\end{cases}
\end{equation}
\end{lemma}

\begin{example}
\label{ex:kn}
For the complete graph $K_n$ on $n$ vertices,
we have $W_{K_n}=0$. On the opposite end, if $\Gamma=\overline{K_n}$
is a discrete graph, then $W_\Gamma=\ker(\delta_2)$.
\end{example}

Let $\RR_{\Gamma}=\RR(G_{\Gamma})\coloneqq \RR(V,K)$ be the resonance scheme
associated to the Koszul module $W_\Gamma$. The underlying set
$\RR_{\Gamma}$ was described in \cite[Theorem 5.5]{PS-mathann}
as a union of coordinate subspaces of $V^{\vee}$.
For a subgraph $\Gamma'\subseteq \Gamma$, we denote by
$V_{\Gamma'}^{\vee}\subseteq V^{\vee}$ the coordinate subspace
spanned by the vertices of $\Gamma'$. With this notation,
\begin{equation}
\label{eq:PS-res-formula}
\RR_{\Gamma}=\bigcup \Bigl\{V^{\vee}_{\Gamma'}: \Gamma'
\mbox{ is a maximally disconnected full subgraph of $\Gamma$}\Bigr\}.
\end{equation}

The next result explains in combinatorial terms
the isotropicity and separability conditions for
resonance components introduced in this paper.

\begin{proposition}
\label{prop:iso-sep-graph}
Let $\Gamma$ be a connected graph,
let $\Gamma'$ be a maximally disconnected full
subgraph, and let ${V'}^{\vee}=V^{\vee}_{\Gamma'}$
be the corresponding component of $\RR_{\Gamma}$. Then,
\begin{enumerate}
\item \label{is1} ${V'}^{\vee}$ is isotropic if and only if $\Gamma'$ is discrete.
\item \label{is2} ${V'}^{\vee}$ is separable
if and only if $\Gamma=\Gamma'*\Gamma''$.
\end{enumerate}
In particular, isotropic implies separable for the resonance varieties of graphs.
\end{proposition}

\begin{proof}
By definition, the linear subspace ${V'}^{\vee}$ is isotropic if
$\bwedge^2 {V'}^{\vee} \subseteq K^{\perp}$, that is, the set
$\{e_i\wedge e_j: (i,j)\in \mathsf{V}'\}$ is contained in
$\{e_i\wedge e_j: (i,j)\in \overline{\mathsf{E}}\}$. This last condition
amounts to $\binom{\mathsf{V}'}{2}\subseteq \overline{\mathsf{E}}$,
which is equivalent to $\Gamma'$ being discrete.

Finally, ${V'}^{\vee}$ is separable if
$K^{\perp} \cap \big(\bwedge^2 {V'}^{\vee} \oplus ({V'}^{\vee}\otimes {V''}^{\vee})\big)
\subseteq \bwedge^2 {V'}^{\vee}$.
This condition is equivalent to $\overline{\mathsf{E}} \cap \big( \binom{\mathsf{V}'}{2} \cup
( \mathsf{V}', \mathsf{V}'') \big)\subseteq  \binom{\mathsf{V}'}{2} $, that is,
$\overline{\mathsf{E}} \cap ( \mathsf{V}', \mathsf{V}'') =\emptyset$, which again
means that $\Gamma=\Gamma'*\Gamma''$.
\end{proof}

We now give two classes of graphs that satisfy
the conditions of Proposition \ref{prop:iso-sep-graph},
and thus, of Theorem \ref{thm:sep->reduced}.
In what follows, we denote by $K_n$ the complete
graph on $n$ vertices.

\begin{example}
\label{ex:join-discrete}
Let $\Gamma=\overline{K_{n_1}} *\cdots * \overline{K_{n_r}}$
be a complete multipartite graph, that is, an iterated join of
discrete graphs. Then $\RR_{\Gamma}$ is
projectively disjoint and isotropic (and hence, separable).
\end{example}

As a concrete example, take the complete graph $K_n=K_1*\cdots *K_1$;
then $\Rproj_{K_n}=\emptyset$.  As another example, consider the
square, $\Gamma=\overline{K_2} * \overline{K_2}$;
then $\Rproj_\Gamma$ is the disjoint union of two (isotropic) lines.

\begin{example}
\label{ex:q-diagonal}
Let $\Gamma$ be a graph obtained from the square with
one diagonal, by taking iterated cones.
Then $\Rproj_{\Gamma}$ consists of a single line, and therefore is
projectively disjoint and isotropic (and hence, separable).
\end{example}

In general, though, the resonance varieties $\RR_{\Gamma}$ are neither
isotropic, nor separable, nor projectively disjoint.  The simplest such
example is as follows.

\begin{example}\label{ex:4paths}
\label{ex:proj reduced non separable}
Let $\Gamma$ be a path on $4$ vertices. Then $V^{\vee}=\spn\{e_1,e_2,e_3,e_4\}$
and $K^{\perp}=\spn\{e_1\wedge e_2,  e_2\wedge e_3, e_3\wedge e_4\}$.
Setting $S=\k[x_1,x_2,x_3,x_4]$, the $S$-module $W_{\Gamma}$ is
presented by the matrix
\[
\Theta_{\Gamma}=
 \bordermatrix{
& {\scriptstyle 13} & {\scriptstyle 14} & {\scriptstyle 24} \cr
{\scriptstyle 123} & - x_2&  0& 0 \cr
{\scriptstyle 124} & 0&-x_2&x_1 \cr
{\scriptstyle 134} & x_4&-x_3&0 \cr
{\scriptstyle 234} & 0& 0&-x_3
} \, .
\]
Note that $\Fitt_0(W_{\Gamma})=(x_2)\cap (x_3) \cap (x_1,x_2^2,x_3^2,x_4)$ is
not reduced, though the annihilator $\Ann(W_{\Gamma})=(x_2)\cap (x_3)$ is reduced.
It follows that $\RR(V,K)=\{x_2=0\} \cup \{x_3=0\}$
is linear and reduced, but neither projectively disjoint, nor isotropic, nor separable.
\end{example}

We refer to \cite{AFRSS} for a comprehensive study of the higher Koszul modules
and the higher resonance schemes associated to monomial ideals corresponding
to arbitrary (finite) simplicial complexes.

\newcommand{\arxiv}[1]
{\texttt{\href{http://arxiv.org/abs/#1}{arXiv:#1}}}
\newcommand{\arx}[1]
{\texttt{\href{http://arxiv.org/abs/#1}{arxiv:}}
\texttt{\href{http://arxiv.org/abs/#1}{#1}}}
\newcommand{\arxx}[2]
{\texttt{\href{https://arxiv.org/abs/#1.#2}{arxiv:#1.}}
\texttt{\href{https://arxiv.org/abs/#1.#2}{#2}}}
\newcommand{\doi}[1]
{\texttt{\href{http://dx.doi.org/#1}{doi:#1}}}
\renewcommand{\MR}[1]
{\href{http://www.ams.org/mathscinet-getitem?mr=#1}{MR#1}}

\bibliographystyle{amsplain}

\end{document}